\documentclass[11pt]{amsart}

\calclayout

\usepackage{mathtools}
\usepackage{amssymb}
\usepackage{amsthm}
\usepackage[T1]{fontenc}
\usepackage[utf8]{inputenc}
\usepackage{lmodern}
\usepackage{inconsolata}
\usepackage{microtype}
\usepackage{enumitem}
\usepackage[colorlinks,unicode]{hyperref}
\usepackage[noabbrev,capitalise,nameinlink]{cleveref}
\usepackage{mleftright} \mleftright
\usepackage{xcolor}

\definecolor{lavender}{rgb}{0.75,0,0.8}
\definecolor{darkblue}{rgb}{0,0,0.85}

\hypersetup{colorlinks=true, citecolor=lavender,linkcolor=lavender,urlcolor=lavender}

\usepackage[bb=boondox]{mathalfa}

\usepackage[backend=biber,style=zbsalph,maxnames=9,backref=true]{biblatex}
\defbibheading{bibliography}[\refname]{\section*{#1}}

\newtheorem{theorem}{Theorem}[section]
\newtheorem{proposition}[theorem]{Proposition}

\newtheorem{question}[theorem]{Question}

\newtheorem{lemma}[theorem]{Lemma}

\theoremstyle{definition}
\newtheorem{definition}[theorem]{Definition}

\DeclarePairedDelimiter{\paren}{(}{)}

\DeclarePairedDelimiter{\set}{\{}{\}}
\DeclarePairedDelimiter{\abs}{\lvert}{\rvert}
\DeclarePairedDelimiter{\Abs}{\lVert}{\rVert}
\DeclarePairedDelimiter{\floor}{\lfloor}{\rfloor}
\DeclarePairedDelimiter{\ceil}{\lceil}{\rceil}

\DeclareMathOperator{\sgn}{sgn}

\let\Vec\mathbf
\newcommand{\perm}{\psi^{\mathrm{max}}}
\newcommand{\calf}{\mathcal{F}}
\newcommand{\eps}{\varepsilon}
\newcommand{\RR}{\mathbb{R}}
\newcommand{\VV}{\mathbf{V}}

\newcommand{\dfn}[1]{\textcolor{darkblue}{\emph{#1}}}

\begin{document}

\title{Ordering Candidates via Vantage Points} 

\author{Noga Alon}
\address{Department of Mathematics, Princeton University, Princeton, NJ 08544, USA and Schools of Mathematical Sciences and Computer Science, Tel Aviv University, Tel Aviv, Israel}
\email{nalon@math.princeton.edu}
\author{Colin Defant}
\address{Department of Mathematics, Harvard University, Cambridge, MA 02139, USA}
\email{colindefant@gmail.com}
\author{Noah Kravitz}
\address{Department of Mathematics, Princeton University, Princeton, NJ 08540, USA}
\email{nkravitz@princeton.edu}
\author{Daniel G. Zhu}
\address{Department of Mathematics, Massachusetts Institute of Technology, Cambridge, MA 02139, USA}
\email{zhd@princeton.edu}

\begin{abstract}
Given an $n$-element set $C\subseteq\RR^d$ and a (sufficiently generic) $k$-element multiset $V\subseteq\RR^d$, we can order the points in $C$ by ranking each point $c\in C$ according to the sum of the distances from $c$ to the points of $V$. Let $\Psi_k(C)$ denote the set of orderings of $C$ that can be obtained in this manner as $V$ varies, and let $\perm_{d,k}(n)$ be the maximum of $\lvert\Psi_k(C)\rvert$ as $C$ ranges over all $n$-element subsets of $\RR^d$. We prove that $\perm_{d,k}(n)=\Theta_{d,k}(n^{2dk})$ when $d \geq 2$ and that $\perm_{1,k}(n)=\Theta_k(n^{4\ceil{k/2}-1})$. As a step toward proving this result, we establish a bound on the number of sign patterns determined by a collection of functions that are sums of radicals of nonnegative polynomials; this can be understood as an analogue of a classical theorem of Warren. We also prove several results about the set $\Psi(C)=\bigcup_{k\geq 1}\Psi_k(C)$; this includes an exact description of $\Psi(C)$ when $d=1$ and when $C$ is the set of vertices of a vertex-transitive polytope. 
\end{abstract}

\maketitle

\section{Introduction}\label{sec:intro}

Let $d \geq 1$ and $n, k \geq 0$ be integers, and consider a set $C=\set{c_1,\ldots,c_n}$ of \dfn{candidate points} in $\RR^d$. Given a multiset $V=\set{v_1,\ldots,v_k}$ of $k$ \dfn{vantage points} in $\RR^d$,
define the function $D_V\colon \RR^d \to \RR$ by $D_V(x) = \sum_{i\in[k]} \Abs{x - v_i}$, where $\Abs{-}$ denotes the Euclidean distance and $[k]\coloneqq\set{1,\ldots,k}$. We say $V$ \dfn{distinguishes} the points in $C$ if the values $D_V(c_1),\ldots,D_V(c_n)$ are distinct. If $V$ distinguishes the points in $C$, then there is a unique permutation $\sigma$ of $[n]$ such that $D_V(c_{\sigma(1)})<\cdots<D_V(c_{\sigma(n)})$; in this case, we say $V$ \dfn{witnesses} the tuple $(c_{\sigma(1)},\ldots,c_{\sigma(n)})$, and we denote this tuple by $\Sigma_V^C$. Throughout this paper, we will identify this tuple with the function $\Sigma_V^C \colon [n] \to C$ that sends $i$ to $c_{\sigma(i)}$, as these two objects clearly contain equivalent information.

Let $\Psi_k(C)$ be the set of tuples $\Sigma_V^C$ witnessed by $k$-element multisets of $\RR^d$ that distinguish the points in $C$, and let $\psi_k(C)=\abs{\Psi_k(C)}$. In other words, $\psi_k(C)$ counts the possible rankings of $C$, where the ranking of a point is determined by the sum of its distances from $k$ vantage points.

The quantity $\psi_1(C)$ was first studied by Good and Tideman \cite{Good1977}, who viewed the points in $C$ as political candidates and the single vantage point $v_1$ as a voter who ranks the candidates based on how far away they are in the Euclidean metric. They proved that \[\psi_1(C) \leq s(n,n)+s(n,n-1)+\cdots+s(n,n-d)\] for every set $C$, where $s(n,r)$ denotes an unsigned Stirling number of the first kind; moreover, they showed that this upper bound is tight. Zaslavsky \cite{Zaslavsky2002} provided a different proof of this inequality using hyperplane arrangements. Carbonero, Castellano, Gordon, Kulick, Ohlinger, and Schmitz \cite{Carbonero2021} continued this line of work by showing that the minimum possible value of $\psi_1(C)$ is $2n-2$; this minimum is independent of the dimension $d$ because it is attained when the points in $C$ are arranged on a line. They also constructed additional point configurations $C$ for which $\psi_1(C)$ attains other values, and they initiated the investigation of $\psi_k(C)$ for larger values of $k$ (with a focus on the case $k=2$). 

Let $\perm_{d,k}(n)$ be the maximum value of $\psi_k(C)$ as $C$ ranges over all $n$-element subsets of $\RR^d$. One of our main results is the following theorem. 
The $k=2$ case asymptotically settles a problem raised in \cite{Carbonero2021}.

\begin{theorem}\label{thm:main}
If $d \geq 2$ and $k \geq 1$ are fixed, then $\perm_{d,k}(n) = \Theta_{d, k}(n^{2dk})$.  If $d=1$ and $k \geq 1$ is fixed, then $\perm_{1,k}(n)=\Theta_k(n^{4\ceil{k/2}-2})$.
\end{theorem}

Given a tuple $\calf=(f_1, \ldots, f_m)$ of real-valued functions on $\RR^N$ and a point $x \in \RR^N$, we obtain the \dfn{sign pattern} $\eps_{\calf}(x)=(\eps_1, \ldots,\eps_m)$, where $\eps_i=\sgn(f_i(x))\in\set{0,1,-1}$.  A sign pattern is \dfn{proper} if its entries are all nonzero. When $\calf$ is a tuple of polynomials of bounded degree, a classical result of Warren \cite{Warren1968} provides an upper bound for the number of distinct proper sign patterns of the form $\eps_{\calf}(x)$ as $x$ varies over $\RR^d$. In \cref{sec:sign_patterns}, we prove an analogue of Warren's theorem which may be of independent interest. This theorem (\cref{thm:sign-patterns-fractional}) bounds the number of proper sign pattern of the form $\eps_{\calf}(x)$ when the functions in $\calf$ are sums of radicals of nonnegative polynomial functions. We then apply this theorem in \cref{sec:upper} to prove the upper bounds $\perm_{d,k}(n)=O_{d,k}(n^{2dk})$ and $\perm_{1,k}(n)=O_k(n^{4\ceil{k/2}-2})$ in \cref{thm:main}. \cref{sec:lower,sec:lowerd1,sec:lowerdg2} are devoted to proving the lower bounds $\perm_{d,k}(n)=\Omega_{d,k}(n^{2dk})$ and $\perm_{1,k}(n)=\Omega_k(n^{4\ceil{k/2}-2})$ in \cref{thm:main}. Our constructions are delicate and technical and proceed by induction on the number of vantage points. 

In \cref{sec:unlimited}, we turn our attention to the set $\Psi(C)=\bigcup_{k\geq 1}\Psi_k(C)$. This is the collection of orderings of $C$ that are witnessed by arbitrarily large (finite) multisets that distinguish the points in $C$. We say a tuple $(x_1,\ldots,x_m)$ of points in $\RR^d$ is \dfn{protrusive} if for every $i\in[m-1]$, the point $x_{i+1}$ is not in the convex hull of $x_1,\ldots,x_i$. A simple argument involving the triangle inequality shows that every tuple in $\Psi(C)$ is protrusive. We show that $\Psi(C)$ is exactly equal to the set of protrusive orderings of the points in $C$ when $d=1$ (\cref{thm:protrusive_d_1}), when $n\leq 4$ (\cref{thm:4-points}), and when $C$ is the set of vertices of a vertex-transitive polytope (\cref{thm:polytope}). In a different direction, we construct a $6$-element set $C\subseteq\RR^2$ such that some protrusive orderings of $C$ are not in $\Psi(C)$. We leave open the problem of determining whether a similar construction exists with only $5$ points.

\section{Sign patterns of sums of radicals}\label{sec:sign_patterns}

It is natural to try to estimate the number of proper sign patterns arising from a tuple $\calf=(f_1, \ldots, f_m)$ of real-valued functions on $\RR^N$. A classical result of Warren gives an upper bound for the case where $f_1, \ldots, f_m$ are polynomials.

\begin{theorem}[{\cite[Theorem~3]{Warren1968}}]\label{thm:warren}
Let $N,m,\Delta$ be positive integers, and let $\calf=(f_1, \ldots, f_m)$, where each $f_i$ is a polynomial in $\RR[x_1,\ldots, x_N]$ of degree at most $\Delta$. Then the number of distinct proper sign patterns of the form $\eps_{\calf}(x)$ for $x\in\RR^N$ is at most
$2(2\Delta)^N\sum_{\ell=0}^N 2^\ell \binom{m}{\ell}$.
\end{theorem}

This theorem has many combinatorial applications (see, e.g., \cite{Alon1995} for some early applications).  We will prove an analogue of Warren's theorem for functions that are sums of radicals of nonnegative polynomials. 

\begin{theorem}\label{thm:sign-patterns-fractional}
Let $N,m,\Delta, r,s$ be positive integers with $r \geq 2$, and let $\calf=(f_1, \ldots, f_m)$, where each $f_i$ is of the form
$f_i=\sum_{j=1}^{r_i} a_{i,j} g_{i,j}^{1/s}$
with $r_i \leq r$ a positive integer, each $a_{i,j}$ a real number, and each $g_{i,j}$ a polynomial in $\RR[x_1, \ldots, x_N]$ of degree at most $\Delta$ such that $g_{i,j}(x) \geq 0$ for all $x \in \RR^N$. Then the number of distinct proper sign patterns of the form $\eps_{\calf}(x)$ for $x\in\RR^N$ is at most
$2(2s^{r-2}\Delta)^N\sum_{\ell=0}^N 2^\ell \binom{m}{\ell}$.
\end{theorem}

Warren deduced \cref{thm:warren} from a topological statement about the connected components of the complement of a real algebraic variety.  Our proof of \cref{thm:sign-patterns-fractional} will follow the same strategy, and we will use Warren's topological statement as a black box.  For a function $p\colon\RR^N \to \RR$, let $\VV(p)\coloneqq\set{x \in \RR^N: p(x)=0}$ denote its zero set.

\begin{lemma}[{\cite[Theorem~2]{Warren1968}}]\label{lem:warren}
Let $N,m,\Delta$ be positive integers, and let $f_1, \ldots, f_m \in \RR[x_1,\ldots, x_N]$ be polynomials of degree at most $\Delta$. Then the set $\RR^N \setminus \bigcup_{i=1}^m \VV(f_i)$ has at most 
$2(2\Delta)^N\sum_{\ell=0}^N 2^\ell \binom{m}{\ell}$
connected components.
\end{lemma}

We will also require the following basic fact about products of ``Galois conjugates.''

\begin{lemma}\label{lem:galois}
Let $r,s \geq 2$ be integers, let $\omega=e^{2\pi i/s}$, and let $\xi_1,\ldots,\xi_r$ be variables.  Then
\[\prod_{0 \leq t_2,\ldots, t_r \leq s-1} (\xi_1+\omega^{t_2} \xi_2+\cdots+\omega^{t_r}\xi_r)\] is a polynomial in $\xi_1^s, \ldots, \xi_r^s$.
\end{lemma}

\begin{proof}
Let
\[X(\xi_1, \ldots, \xi_r)=\prod_{0 \leq t_2,\ldots, t_r \leq s-1} (\xi_1+\omega^{t_2} \xi_2+\cdots+\omega^{t_r}\xi_r).\]
We claim that for each $1 \leq j \leq r$, we have
\[X(\xi_1, \ldots, \xi_r)=X(\xi_1, \ldots, \xi_{j-1}, \omega \xi_j, \xi_{j+1}, \ldots, \xi_r).\]
Indeed, for $j=1$, we can write \[\omega\xi_1+\omega^{t_2} \xi_2+\cdots+\omega^{t_r}\xi_r=\omega(\xi_1+\omega^{t_2-1} \xi_2+\cdots+\omega^{t_r-1}\xi_r).\]
So
\begin{align*}
X(\omega\xi_1, \xi_2, \ldots, \xi_r) &=\omega^{s^{r-1}} \prod_{-1 \leq t_2,\ldots, t_r \leq s-2} (\xi_1+\omega^{t_2} \xi_2+\cdots+\omega^{t_r}\xi_r)\\
 &=\prod_{0 \leq t_2,\ldots, t_r \leq s-1} (\xi_1+\omega^{t_2} \xi_2+\cdots+\omega^{t_r}\xi_r)=X(\xi_1, \ldots, \xi_r), 
\end{align*}
where in the second equality we used the fact that $\omega^{-1}=\omega^{s-1}$.
When $j>1$, we may assume that $j=r$ and write
\[\xi_1+\omega^{t_2} \xi_2+\cdots+\omega^{t_r}(\omega\xi_r)=\xi_1+\omega^{t_2} \xi_2+\cdots+\omega^{t_r+1}\xi_r;\]
we then conclude as in the $j=1$ case, and this establishes the claim.

For $\alpha=(\alpha_1, \ldots, \alpha_r) \in \mathbb{Z}_{\geq 0}^r$, write $\xi^\alpha\coloneqq\xi_1^{\alpha_1} \cdots \xi_r^{\alpha_r}$.  Then there are unique constants $b_\alpha$ (with $\alpha$ satisfying $\alpha_1+\cdots+\alpha_r=s^{r-1}$) such that
\[X(\xi_1, \ldots, \xi_r)=\sum_{\alpha} b_\alpha \xi^\alpha.\]
Suppose $\alpha=(\alpha_1, \ldots, \alpha_r)$ is such that $\alpha_i$ is not a multiple of $s$.  Then the coefficient of $\xi^\alpha$ in $X(\xi_1, \ldots, \xi_{j-1}, \omega \xi_j, \xi_{j+1}, \ldots, \xi_r)$ is $b_\alpha \omega^{\alpha_i}$.  But, by the claim, this also equals $b_\alpha$.  Since $\omega^{\alpha_i} \neq 1$, we conclude that $b_\alpha=0$.
\end{proof}

We can now prove \cref{thm:sign-patterns-fractional}.

\begin{proof}[Proof of \cref{thm:sign-patterns-fractional}]
Given $x\in\RR^N$ and a small positive real number $\delta$, let \[f_{i,\delta}(x)=\sum_{j=1}^{r_i}a_{i,j}(g_{i,j}(x)+\delta)^{1/s};\] that is, $f_{i,\delta}$ is obtained by replacing each $g_{i,j}$ with $g_{i,j} + \delta$ in the definition of $f_i$. Let $\calf_\delta=(f_{1,\delta},\ldots,f_{m,\delta})$. For each fixed $x \in \RR^N$, the quantity $f_{i,\delta}(x)$ is continuous in $\delta$. Therefore, if $\eps_\calf(x)$ is proper and $\delta$ is sufficiently small, then $\eps_\calf(x)=\eps_{\calf_\delta}(x)$. It follows that if $\delta$ is sufficiently small, then the number of proper sign patterns cannot decrease when we replace $\calf$ with $\calf_\delta$. This shows that it suffices to prove the theorem when $g_{i,j}(x) > 0$ for all $x \in \RR^N$; we will henceforth assume this.

If some $f_i$ is the zero function, then no proper sign patterns are achieved, so the result is obvious. Otherwise, we claim that for each $i$, we can find a nonzero polynomial $\tilde f_i \in \RR[x_1,\ldots,x_N]$ of degree at most $s^{r-2} \Delta$ such that $\VV(f_i) \subseteq \VV(\tilde f_i)$. To see how this claim finishes the proof, note that by applying \cref{lem:warren}, we find that $\RR^N \setminus \bigcup_{i=1}^m \VV(\tilde f_i)$ has at most 
$2(2s^{r-2}\Delta)^N\sum_{k=0}^N 2^k \binom{m}{k}$
connected components. Since the sign pattern $\eps_\calf(x)$ must be constant on every such connected component, we conclude that $\eps_\calf(x)$ assumes at most $2(2s^{r-2}\Delta)^N\sum_{k=0}^N 2^k \binom{m}{k}$ proper values as $x$ varies over $\RR^N \setminus \bigcup_{i=1}^m \VV(\tilde f_i)$. For each $\eps\in\set{1,-1}^m$, the set $\set{x : \eps_\calf(x) = \eps}$ is open, so it cannot be contained in the nowhere-dense set $\bigcup_{i=1}^m \VV(\tilde f_i)$. Hence, there are at most $2(2s^{r-2}\Delta)^N\sum_{k=0}^N 2^k \binom{m}{k}$ proper sign patterns of the form $\eps_\calf(x)$.

Now we fix an $i$ and prove the claim. First, take a minimal subset $G$ of the set $\set{g_{i,1}^{1/s}, \ldots, g_{i,r_i}^{1/s}}$ such that $f_i$ is in the linear span of $G$; without loss of generality, we may choose the ordering of the polynomials $g_{i,1},\ldots,g_{i,r_i}$ so that $G = \set{g_{i,1}^{1/s},\ldots,g_{i,r'_i}^{1/s}}$ and write $f_i = \sum_{j=1}^{r'_i} a'_{i,j} g_{i,j}^{1/s}$. Note that the elements of $G$ are linearly independent and that $a'_{i,j} \neq 0$ for all $j$. Since $f_i$ is not the zero function, $r'_i > 0$. Moreover, if $r'_i = 1$, $f_i$ is never zero, in which case we may set $\tilde f_i  = 1$. It remains to consider the case where $r'_i \geq 2$.

Let $\xi_j=a'_{i,j}g_{i,j}^{1/s}$, $\omega = e^{2\pi i/s}$, and
\[
\tilde{f}_i=\prod_{0 \leq t_2,\ldots, t_{r'_i} \leq s-1} (\xi_1+\omega^{t_2} \xi_2+\cdots+\omega^{t_{r'_i}}\xi_{r'_i}).
\]
\cref{lem:galois} guarantees that $\tilde{f}_i$ is a polynomial in $g_{i,1}, \ldots, g_{i,r'_i}$ and hence is a polynomial in $\RR[x_1,\ldots,x_N]$.  The degree of $\tilde{f}_i$ in $\xi_1,\ldots,\xi_{r_i}$ is $s^{r'_i-1}$, so its degree in $g_{i,1}, \ldots, g_{i,r'_i}$ is $s^{r'_i-2}$. Hence, the degree of $\tilde f_i$ in $x_1, \ldots, x_N$ is at most $s^{r'_i-2}\Delta \leq s^{r-2}\Delta$.

As it is obvious that $\VV(f_i) \subseteq \VV(\tilde f_i)$, it suffices to show that  $\tilde{f}_i$ is not the zero polynomial. To see this, note that since $g_{i,1}^{1/s},\ldots,g^{1/s}_{i,r'_i}$ are linearly independent and $a'_{i,j} \neq 0$ for all $j \in [r'_i]$, the function $\xi_1+\omega^{t_2} \xi_2+\cdots+\omega^{t_{r'_i}}\xi_{r'_i}\colon\RR^N\to\mathbb C$ cannot be the zero function for any choice of $t_2,\ldots,t_{r'_i}$. Moreover, this function is analytic because we assumed that $g_{i,j}(x) > 0$ for all $j\in[r_i]$ and $x \in \RR^N$. This implies that $\tilde f_i$ cannot be zero, since the product of nonzero complex-valued analytic functions cannot be zero.
\end{proof}
We remark that the same proof yields a similar bound in the more general case where the fractional power appearing in the definition
of the function $f_i$ is $s_i$, and the integers $s_i$ are not necessarily all equal. We omit the details since this more general statement is not needed here.

The following example shows that the functions $f_i$ considered in \Cref{thm:sign-patterns-fractional} can have many sign changes even when $N=1$ and $\Delta=s=2$ are fixed.

\begin{proposition}
\label{prop:noga}
Let $\ell$ be a positive integer, and let $0<\delta<2/\ell$ be a real number.  For $\Vec a=(a_1, \ldots, a_\ell) \in \set{-1,1}^\ell$, define $f_{\Vec a}\colon\RR\to\RR$ by
\[
f_{\Vec a}(x)\coloneqq\sum_{i=1}^\ell a_i \paren*{\sqrt{(x-i)^2 +\delta^2} - \sqrt{(x-i)^2}}.
\]
Then $\sgn(f_{\Vec a}(j))=a_j$ for all $j\in[\ell]$. 
\end{proposition}

\begin{proof}
Notice that the $i=j$ summand in the definition of $f_{\Vec a}(j)$ equals $\delta a_j$.  For $i \neq j$, Bernoulli's inequality tells us that
\[\abs*{\sqrt{(j-i)^2 +\delta^2} - \sqrt{(j-i)^2}}<\frac{\delta^2}{2\abs{j-i}} \leq \frac{\delta^2}{2}.\]
Hence, the total contribution to $f_{\Vec a}(j)$ of the terms with $i \neq j$ is at most $(\ell-1) \delta^2/2<\delta$. It follows that the $i=j$ term dominates the sum, so $\sgn(f_{\Vec a}(j))=a_j$.
\end{proof}
In particular, if we let $\ell = 2^m$ for some positive integer $m$ and choose $\Vec a_1,\ldots,\Vec a_m$ appropriately, then the tuple of functions $(f_{\Vec a_1},\ldots,f_{\Vec a_m})$ can take on all $2^m$ proper sign patterns of length $m$. Thus we can achieve $2^m$ proper sign patterns even while fixing $N = 1$ and $\Delta = s = 2$, at the expense of letting $r = 2^{m+1}$ grow exponentially. Somewhat informally, this implies that the bound in \cref{thm:sign-patterns-fractional} must have at least a linear dependence on $r$ if it is to depend polynomially on $m$. This is very far from our bound (which is exponential in $r$), and it would be interesting to close this gap.

A simpler, though still interesting, problem in this direction concerns the case where we fix $N = m = 1$ and $\Delta = s = 2$ and consider a function $f \colon \RR \to \RR$ which is a linear combination of $r$ square roots of everywhere-positive quadratic polynomials. Instead of considering the number of proper sign patterns of $f$, which is obviously bounded above by $2$, we instead consider the number of connected components of $\RR \setminus \VV(f)$. Using a similar ``multiplication by conjugates'' trick, this can be bounded above by $2^{r-1} + 1$. On the other hand, we can achieve $2r-1$ connected components by letting
\[f(x) = 1 + \sum_{i=1}^{r-1} (-1)^i a_i^{1/10} \paren*{\sqrt{x^2 + a_i^2} - a_i}\]
for a sequence $0 < a_1 < \cdots < a_{r-1}$ that grows extremely quickly. Again we have a linear lower bound and an exponential upper bound; it would be interesting to narrow the gap.

\section{The upper bound in \texorpdfstring{\cref{thm:main}}{Theorem \ref{thm:main}}}\label{sec:upper}

Using the tools from the previous section, we can quickly establish the upper bound in \Cref{thm:main}.

\begin{lemma}\label{lem:main-upper}
If $d \geq 1$ and $k \geq 1$, then $\perm_{d,k}(n)=O_{d,k}(n^{2dk})$.
\end{lemma}

\begin{proof}
Let $C=\set{c_1, \ldots, c_n}\subseteq \RR^d$ be a set of $n$ (distinct) candidate points.  Consider the multiset of vantage points $V=\set{v_1, \ldots, v_k} \subseteq \RR^d$ as a variable.  For each of the $\binom{n}{2}$ choices of $1 \leq i<j \leq n$, define the function
\[f_{i,j}(V)\coloneqq D_V(c_i)-D_V(c_j)=\sum_{\ell=1}^k \Abs{c_i-v_\ell}-\sum_{\ell=1}^k \Abs{c_j-v_\ell}.\]
Writing each vantage point in coordinates as $v_\ell=(v_\ell^{(1)},\ldots, v_\ell^{(d)})$, we see that the functions $f_{i,j}$ satisfy the hypotheses of \cref{thm:sign-patterns-fractional} with the parameters $N=dk$, $m=\binom{n}{2}$, $\Delta=2$, $r=2k$, and $s=2$.  We conclude that the functions $f_{i,j}$ witness at most
\[2(2^{2k})^{dk} \sum_{\ell=0}^{dk} 2^\ell \binom{\binom{n}{2}}{\ell}=O_{d,k}(n^{2dk})\]
distinct proper sign patterns.  The signs of the functions $f_{i,j}$ fully determine the relative sizes of the quantities $D_V(c_i)$, so $\psi_k(C)=O_{d,k}(n^{2dk})$.  Since this holds for every choice of $C$, we conclude that $\perm_{d,k}(n)=O_{d,k}(n^{2dk})$, as desired.
\end{proof}

\Cref{lem:main-upper} settles the upper bound for all $k$ when $d \geq 2$ and for odd $k$ when $d = 1$.  The following lemma handles the remaining cases.
\begin{lemma}
If $k \geq 2$ is even, then $\perm_{1,k}(n)=O_{k}(n^{2k-2})$.
\end{lemma}
\begin{proof}
Given a multiset $V = \set{v_1,\ldots,v_k}$ with $v_1 \leq v_2 \leq \cdots \leq v_k$, let \[\overline V = \set*{v_1,\ldots,v_{k/2-1},\frac{v_{k/2}+v_{k/2+1}}{2},\frac{v_{k/2}+v_{k/2+1}}{2},v_{k/2+2},\ldots,v_k}.\] We claim that if $V$ distinguishes $C$, then $\overline V$ also distinguishes $C$ and we have  $\Sigma^C_{\overline V} = \Sigma^C_V$. To see this, note that $D_V$ achieves its minimum precisely on the interval $[v_{k/2}, v_{k/2+1}]$. Moreover, $D_{\overline V}(x) = D_V(x)$ for all $x \notin [v_{k/2},v_{k/2+1}]$, while $D_{\overline V}(x) \leq D_V(x)$ for all $x \in [v_{k/2},v_{k/2+1}]$. Therefore, the values in $(D_{\overline V}(c))_{c \in C}$ are the same as in $(D_V(c))_{c \in C}$, except that the minimum may decrease. This cannot change the relative order.

As a result, to determine $\Psi_1(C)$, one needs to consider $\Sigma^C_V$ only for multisets $V$ of $k$ points in which two points are identical. Arguing as in the proof of \cref{lem:main-upper} gives $\perm_{1,k}(n) = O_k(n^{2(k-1)})$, as desired.
\end{proof}

\section{The lower bound in \texorpdfstring{\cref{thm:main}}{Theorem \ref{thm:main}}: general techniques}\label{sec:lower}
\subsection{Overview of the proof strategy} \label{subsec:lowerintro}
Before diving into the technical details of the proof of~\cref{thm:main}, we will give a qualitative high-level description of our strategy.

It is reasonable to expect that a generic $n$-element set $C\subseteq \RR^d$ will have $\psi_{k}(C)$ within a constant factor of $\perm_{d,k}(n)$. The difficulty in proving the lower bound in \cref{thm:main} thus lies in finding a set $C$ that both ``behaves generically'' and has enough structure to be analyzed.  To achieve this, our construction of $C$ will contain features on many different scales.  When $d \geq 2$, this will let us use asymptotic estimates such as $\sqrt{1+R^2}\approx R$ and $\sqrt{R^2 + aR} - R \approx a/2$ for large $R$, which simplify the square roots inherent in Euclidean distances.

For a fixed $d$, our proof proceeds by inductively turning a construction with $k$ vantage points into a construction with $k+2$ vantage points.  When $d=1$, our base cases are $k=1$ and $k=2$; the result for the former is already known, and the result for the latter is an immediate consequence.  When $d \geq 2$, our base cases $k = 0$ and $k = 1$ are trivial and already known, respectively. The inductive step takes a set $C'$ from our inductive hypothesis and adds two carefully-chosen sets $C_1$ and $C_2$ ``flanking'' $C'$ such that $C'$ is located roughly halfway between $C_1$ and $C_2$. We make the scale of $C_1$ and $C_2$ much larger than the scale of $C'$, and we make the separation distances between $C'$, $C_1$, and $C_2$ even larger.

We place two vantage points $u_1$ and $u_2$ near $C_1$ and $C_2$, respectively, and place $k$ additional vantage points $v_1, \ldots, v_k$ close to $C'$.  The quantity $\Abs{c - u_1} + \Abs{c - u_2}$ will be essentially constant (in fact, exactly constant when $d=1$) on $C'$ due to the large separation distances, so the relative order of the points of $C'$ will be entirely determined by the $k$ vantage points near $C'$.  At the same time, for each $c\in C_1\cup C_2$, since the scale of $C_1$, $C_2$, $u_1$, and $u_2$ exceeds that of $C'$ and $v_1, \ldots, v_k$, the variation in the quantity $\Abs{c - u_1} + \Abs{c - u_2}$ (as $u_1$ and $u_2$ vary) will vastly exceed the variation in the quantity $\sum_i \Abs{c-v_i}$ (as the $v_i$'s vary), so the relative order of the points in $C_1 \cup C_2$ will be entirely determined by the vantage points $u_1$ and $u_2$. As a result, the relative orderings of the sets $C'$ and $C_1 \cup C_2$ can be determined independently, and the total number of orderings of $C' \cup C_1 \cup C_2$ will be at least the product of the numbers of orderings of the two component parts.

To formalize the notion of large separations, we define ``effective distance functions'' that control the relative order of the points in $C_1 \cup C_2$ in the limit where the scales of $C_1$ and $C_2$, as well as their separation distance, go to infinity. This perspective allows us to reduce the original problem to a two-``vantage-point'' subproblem involving the effective distance functions. For any given individual solution to the subproblem, we will be able to complete the inductive step by choosing sufficiently large scales and separation distances, but this limiting process can be forgotten entirely when solving the subproblem itself.  After this stage, the proofs for $d=1$ and for $d \geq 2$ diverge.

For $d=1$, our effective distance functions are piecewise linear and can be analyzed manually. The argument for $d \geq 2$ is more complicated.  After the initial reduction via large separations, our next reduction scales $C_2$ to be much larger than $C_1$; this further simplifies the effective distance functions.  We will take $C_1$ and $C_2$ to be two (scaled) copies of a $(d-1)$-dimensional configuration $C^*$ such that $\psi_1(C^*) = \Omega_d(n^{2(d-1)})$; they will be located on two hyperplanes $H_1$ and $H_2$ orthogonal to the separation axis. The projection of $u_i$ onto $H_i$ will determine the relative ordering of the points in $C_i$ for each $i$, and then the orthogonal distances from $u_1, u_2$ to $H_1, H_2$ will determine the interleaving of points from $C_1$ and points from $C_2$ in the ordering of $C_1 \cup C_2$.  Thanks to our prior simplification work, this interleaving process can be analyzed almost exactly, and we will show that generically there are $\Theta(n^4)$ possible interleavings. Combined with $\Theta_d(n^{2(d-1)})$ orderings of each of $C_1$ and $C_2$, this will give $\Theta_d(n^{4d})$ total orderings of $C_1 \cup C_2$. A schematic of the entire construction is shown in \cref{fig:constructionschematic}. 

\begin{figure}[tbp]
\begin{center}{\includegraphics[height=5.44cm]{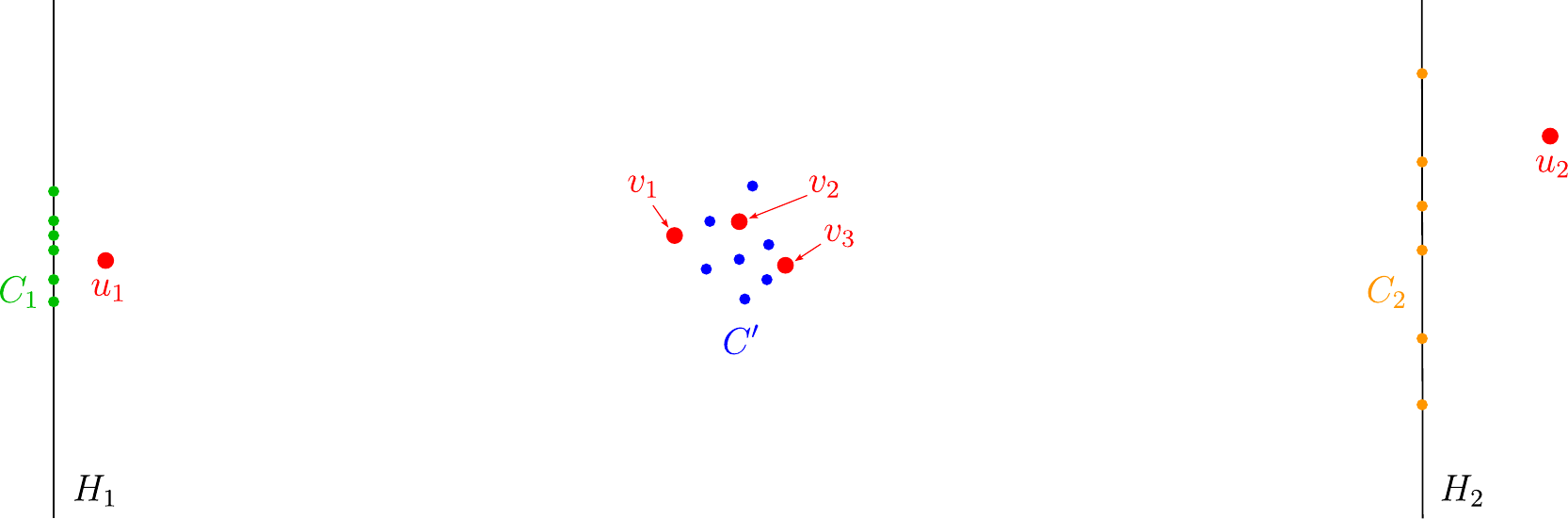}}
\end{center}
\caption{A schematic illustration of our general recursive approach for $d \geq 2$. Dots representing vantage points are large and red.}
\label{fig:constructionschematic}
\end{figure}

\subsection{Pairs of flanking points} 
We begin by formalizing the notion of an infinite separation limit through the definition of effective distance functions. For the remainder of the proof, write $e_1 \coloneqq (1,0, 0,\ldots)$ for the first standard basis vector in $\RR^d$.

\begin{definition}
Let $k$ be a nonnegative integer, let $\hat U = (\hat u_1, \hat u_2) \in \RR^d \times \RR^d$ be an ordered pair of points, and let $\hat c\in\RR^d$ be a point.  Define
\[
\hat{D}^1_{k,\hat U}(\hat c) \coloneqq \Abs{\hat c - \hat u_1} + ((k+1)\hat c + \hat u_2) \cdot  e_1 \quad\text{and}\quad
\hat{D}^2_{k,\hat U}(\hat c) \coloneqq \Abs{\hat c - \hat u_2} + ((k+1)\hat c + \hat u_1) \cdot e_1.
\]
\end{definition}

To motivate these definitions,
we consider the multiset of vantage points \[V = \set{\hat u_1 + Re_1, -\hat u_2 - Re_1, \underbrace{0,\ldots,0}_{k\text{ points}}}\] and the candidate points $c_1 = \hat c_1 + Re_1$ and $c_2 = -\hat c_2 - Re_1$, where $R$ is a large positive real number.  With $\hat u_1, \hat u_2, \hat c_1, \hat c_2$ held constant, we have
\[\hat{D}_{k,\hat U}^1(\hat c_1) = \lim\limits_{R \to \infty} (D_{V}(c_1) - (k+1)R) \quad\text{and}\quad
\hat{D}_{k,\hat U}^2(\hat c_2) = \lim\limits_{R \to \infty} (D_{V}(c_2) - (k+1)R).\]
So, if we care about the relative sizes of quantities of the form $D_V(c_1), D_V(c_2)$ in the regime where $R$ is large, then it suffices to understand the relative sizes of the quantities $\hat{D}_{k,\hat U}^1(\hat c_1), \hat{D}_{k,\hat U}^2(\hat c_2)$.  We now define an analogue of $\psi_k(C)$ for these effective distances.  In what follows, we write $A \sqcup B$ to denote the disjoint union of the sets $A,B$ (even if $A,B$ have nonempty intersection as sets).

\begin{definition} \label{def:hatpsi}
Let $k$ be a nonnegative integer, and let $\hat C_1$ and $\hat C_2$ be sets of points in $\RR^d$. Given $\hat U \in \RR^d \times \RR^d$, let $\hat D_{k,\hat U} \colon \hat C_1 \sqcup  \hat C_2 \to \RR$ be the function that equals $\hat D^1_{k,\hat U}$ on $\hat C_1$ and equals  $\hat D^2_{k,\hat U}$ on $\hat C_2$. Let $\hat \Sigma_{\hat U}^{\hat C_1,\hat C_2}$ be the function $[\abs{\hat C_1} + \abs{\hat C_2}] \to \hat C_1 \sqcup\hat C_2$ such that $\hat D_{k,\hat V} \circ \hat \Sigma_{\hat U}^{\hat C_1,\hat C_2}$ is increasing (if such a function exists), and let $\hat \Psi_k(\hat C_1, \hat C_2)$ be the set of all such $\hat \Sigma_{\hat U}^{\hat C_1,\hat C_2}$. Finally, let $\hat \psi_k(\hat C_1,\hat C_2) = \abs{\hat \Psi_k(\hat C_1,\hat C_2)}$.
\end{definition}

The following three lemmas capture the main steps of our inductive argument.  The first is for all $d$, the second is for $d=1$, and the third is for $d \geq 2$.
 
\begin{lemma} \label{lem:lowerinduction}
Let $d \geq 1$, $k\geq 0$, and $m \geq 0$, with the additional constraint that $m=0$ if $k=0$. Then, for sets $\hat C_1, \hat C_2 \subseteq \RR^d$, we have
\[\perm_{d,k+2}(m + \abs{\hat C_1} + \abs{\hat C_2}) \geq \hat \psi_k(\hat C_1, \hat C_2) \perm_{d,k}(m).\]
\end{lemma}
\begin{lemma} \label{lem:biexistd1}
For $k,m \geq 1$, there exist sets $\hat C_1, \hat C_2 \subseteq \RR$ of size $2m$ such that ${\hat \psi_k(\hat C_1, \hat C_2) \geq m^4}$.
\end{lemma}
\begin{lemma} \label{lem:biexist}
For $d \geq 2$, $k \geq 0$, and $m \geq 1$, there exist sets $\hat C_1, \hat C_2 \subseteq \RR^d$ of size $m$ such that $\hat \psi_k(\hat C_1, \hat C_2) = \Omega_{d}(m^{4d})$.
\end{lemma}

Before proving these lemmas, let us see how they imply the desired estimates on $\perm_{d,k}(n)$ for fixed $d,k$.  We proceed by induction on $k$. For the inductive step, we note that combining \cref{lem:biexistd1,lem:biexist} yields that for all $d \geq 1$ and $k \geq 1$, there exist sets $\hat C_1, \hat C_2$ of size at most $n/3$ with $\hat \psi_k(\hat C_1, \hat C_2) = \Omega_{d}(n^{4d})$, so after applying \cref{lem:lowerinduction} we find that
\[\perm_{d,k+2}(n) \geq \psi_{d,k+2}(\floor{n/3} + \abs{\hat C_1} + \abs{\hat C_2}) \geq \Omega_{d}(n^{4d}) \perm_{d,k}(\floor{n/3}).\]
Therefore it suffices to show the base cases $k = 1,2$.

The base case $k = 1$, i.e., that $\perm_{d,1}(n) = \Omega_d(n^{2d})$ for all $d \geq 1$, follows from the result of \cite{Good1977, Zaslavsky2002} discussed in \Cref{sec:intro}. In the $k = 2$ case, we wish to prove that $\perm_{d,2}(n)$ is $\Omega(n^2)$ if $d = 1$ and $\Omega_d(n^{4d})$ otherwise. In the $d = 1$ case, this follows from the fact that $\perm_{1,2}(n) \geq \perm_{1,1}(n)$, which can be easily proven by putting the two vantage points at the same place. In the $d \geq 2$ case, we use \cref{lem:biexist} to find sets $\hat C_1$ and $\hat C_2$ of size $\floor{n/2}$ such that $\hat \psi_k(\hat C_1,\hat C_2) = \Omega_d(n^{4d})$. Then, by \cref{lem:lowerinduction} we have
\[\perm_{d,2}(n) \geq \perm_{d,2}(2\floor{n/2}) \geq \hat \psi_k(\hat C_1,\hat C_2) = \Omega_d(n^{4d}),\]
where we use the trivial fact that $\perm_{d,0}(0) = 1$.

We now dispose of \cref{lem:lowerinduction}; the proofs of \cref{lem:biexistd1,lem:biexist} will occupy the following two sections.

\begin{proof}[Proof of \cref{lem:lowerinduction}]
Let $C'$ be such that $\psi_k(C') = \perm_{d,k}(m)$. Then, for a positive real number $R$, let $C_1 = R^3e_1 + R\hat C_1$, $C_2 = -R^3 e_1 - R\hat C_2$, and $C = C' \cup C_1 \cup C_2$. Since $C'$, $C_1$, and $C_2$ are disjoint for $R$ sufficiently large, it suffices to show that
\[\psi_{k+2}(C) \geq \psi_k(C') \hat \psi_k(\hat C_1, \hat C_2)\]
for all sufficiently large $R > 0$.

Given two functions $\sigma' \colon [\abs{C'}] \to C'$ and $\hat \sigma \colon [\abs{\hat C_1} + \abs{\hat C_2}] \to \hat C_1 \sqcup \hat C_2$, define $\sigma' \boxplus \hat \sigma \colon [\abs{C}] \to C$ to be the function given by
\[
(\sigma' \boxplus \hat \sigma)(i) = \begin{cases}
\sigma'(i) & i \leq \abs{C'} \\
\phi(\hat \sigma(i - \abs{C'})) & i > \abs{C'},
\end{cases}\]
where $\phi$ is the natural map $\hat C_1 \sqcup \hat C_2 \to C_1 \cup C_2$.  In other words, if $\sigma'$ is an ordering of $C'$ and $\hat \sigma$ is an ordering of $\hat C_1 \sqcup \hat C_2$, then $\sigma' \boxplus \hat \sigma$ is the ordering of $C$ that concatenates $\sigma'$ and $\hat{\sigma}$ (when viewed as tuples).
Now take an ordered pair $\hat U = (\hat u_1, \hat u_2) \in \RR^d \times \RR^d$ and a multiset $V' \subseteq \RR^d$ of size $k$, and let $u_1 = R^3 e_1 + R\hat u_1$, $u_2 = -R^3 e_1 - R\hat u_2$, and $V = V' \cup \set{u_1,u_2}$. The lemma will follow from the statement that $\Sigma_{V}^C = \Sigma_{V'}^{C'} \boxplus \hat \Sigma_{\hat U}^{\hat C_1,\hat C_2}$ for all sufficiently large $R$. 

We now estimate $D_V(c)$ in three different ways, according to which of $C', C_1, C_2$ contains $c$.  First, for $c \in C'$, we have
\[D_{V}(c) = D_{V'}(c) + \Abs{R^3e_1 + R\hat u_1 - c} + \Abs{R^3 e_1 + R\hat u_2 + c}.\]
Since $R^3 e_1 + R\hat u_1 - c=[R^3 + (R\hat u_1-\hat c) \cdot e_1]e_1+O(R)$, we have
\begin{align*}\Abs{R^3e_1 + R\hat u_1 - c} &= \sqrt{(R^3 + (R\hat u_1-c) \cdot e_1)^2 + O(R^2)} \\ &= (R^3 + (R\hat u_1-c) \cdot e_1)\sqrt{1 + O(R^{-4})} \\ &= R^3 + (R\hat u_1-c) \cdot e_1 + O(R^{-1}).
\end{align*}
Similarly,
\[\Abs{R^3 e_1 + R\hat u_2 + c} = R^3 + (R\hat u_2 + c)\cdot e_1 + O(R^{-1}),\]
so
\[
D_{V}(c) = 2R^3 + R(\hat u_1 + \hat u_2) \cdot e_1 + D_{V'}(c) + O(R^{-1}).
\]
Now, for $\hat c \in \hat C_1$, we compute
\begin{align*}
D_{V}(R^3e_1 + R\hat c) &= \Abs{R^3 e_1 + R\hat c - R^3e_1 - R\hat u_1} + \Abs{R^3 e_1 + R\hat c + R^3 e_1 + R\hat u_2} + D_{V'}(R^3e_1 + Rc) \\
&= R\Abs{ \hat c - \hat u_1} + \Abs{2R^3 e_1 + R(\hat c + \hat u_2)} + k\Abs{R^3 \hat e_1 + R\hat c} + O(1).
\end{align*}
As above, we can estimate \[\Abs{2R^3e_1 + R(\hat c + \hat u_2)}=2R^3 + R(\hat c + \hat u_2) \cdot e_1 + O(R^{-1})\quad\text{and}\quad\Abs{R^3 e_1 + R\hat c} = R^3 e_1 + R\hat c \cdot e_1 + O(R^{-1}).\]  So
\[
D_{V}(R^3e_1 + R\hat c) = (k+2)R^3 + R\hat D^1_{k,\hat U}(\hat c) + O(1).
\]
Similarly, for $\hat c \in \hat C_2$, we have
\[
D_{V}(-R^3e_1 - R\hat c) = (k+2)R^3 + R\hat D^2_{k,\hat U}(\hat c) + O(1).
\]
Finally, comparing these estimates, we see that for large enough $R$, the quantities $D_V(c)$ for $c \in C'$ are all smaller than the quantities $D_V(c)$ for $c \in C_1 \cup C_2$ (here using the assumption that $C'$ is empty if $k=0$); moreover, the relative order of the points in $C'$ is $\sigma'$, and the relative order of the points in $C_1 \cup C_2$ is $\phi \circ \hat \sigma$.
\end{proof}
\section{The lower bound in \texorpdfstring{\cref{thm:main}}{Theorem \ref{thm:main}}: \texorpdfstring{$d = 1$}{d = 1}} \label{sec:lowerd1}
To prove \cref{thm:main} when $d=1$, it remains only to prove \cref{lem:biexistd1}.
\begin{proof}[Proof of \cref{lem:biexistd1}]
Let $a_1,\ldots,a_m$ and $b_1,\ldots,b_m$ be real numbers that are sufficiently generic that the $m^2$ differences $a_i - b_j$ for $i, j \in [m]$ are all distinct (for instance, $a_i = i/m$ and $b_i = i$ will work). Now let $R$ be a real number, and define the sets
\begin{align*}
\hat C_1 &= \set{k(R+a_1),\ldots,k(R+a_m),k(3R + a_1),\ldots,k(3R + a_m)}, \\
\hat C_2 &= \set{(k+2)(R+b_1),\ldots,(k+2)(R+b_m),2(k+2)R+k(R+b_1),\ldots,2(k+2)R+k(R+b_m)}.
\end{align*}
We claim that for large $R$, we have $\hat \psi_k(\hat C_1, \hat C_2) \geq m^4$.

For each fixed choice of $w_1,w_2 \in \RR$, let $\hat U = (k(k+2)w_1, 2(k+2)R+k(k+2)w_2)$. Then, for sufficiently large $R$ (depending on $w_1,w_2$), we may compute
\begin{align*}
\hat D^1_{k,\hat U}(k(R + a_i)) &= k(k+2)(R + a_i)-k(k+2)w_1+2(k+2)R+k(k+2)w_2, \\
\hat D^1_{k,\hat U}(k(3R + a_i)) &= k(k+2)(3R + a_i)-k(k+2)w_1+2(k+2)R+k(k+2)w_2, \\
\hat D^2_{k,\hat U}((k+2)(R + b_i)) &= k(k+2)(R + b_i)+k(k+2)w_1+2(k+2)R+k(k+2)w_2, \\
\hat D^2_{k,\hat U}(2(k+2)R+k(R+b_i)) &= k(k+2)(3R + b_i)+k(k+2)w_1+2(k+2)R-k(k+2)w_2.
\end{align*}
Canceling many common terms, the relative order of the $2m$ numbers $\hat D^1_{k,\hat U}(k(R + a_i))$ and ${\hat D^2_{k,\hat U}((k+2)(R + b_i))}$ is given by the relative order of the numbers
\[a_1 - w_1, \ldots, a_m - w_1, b_1 + w_1, \ldots, b_m + w_1.\]
Since the differences $a_i - b_j$ are distinct, these numbers assume $m^2 + 1 \geq m^2$ different orderings as $w_1$ varies.

At the same time, the relative order of the numbers $\hat D^1_{k,\hat U}(k(3R + a_i))$ and $\hat D^2_{k,\hat U}(2(k+2)R+k(R+b_i))$ is given by the relative order of
\[a_1 - (w_1 - w_2), \ldots, a_m - (w_1 - w_2), b_1 + (w_1 - w_2),\ldots, b_m + (w_1 - w_2).\]
As $w_1 - w_2$ varies, these numbers also assume at least $m^2$ orderings. Since $w_1$ and $w_1 - w_2$ can vary independently, we obtain at least $m^4$ distinct elements of $\hat \Psi_k(\hat C_1, \hat C_2)$.  This concludes the proof.
\end{proof}

\begin{figure}[tbp]
\begin{center}{\includegraphics[height=5.574cm]{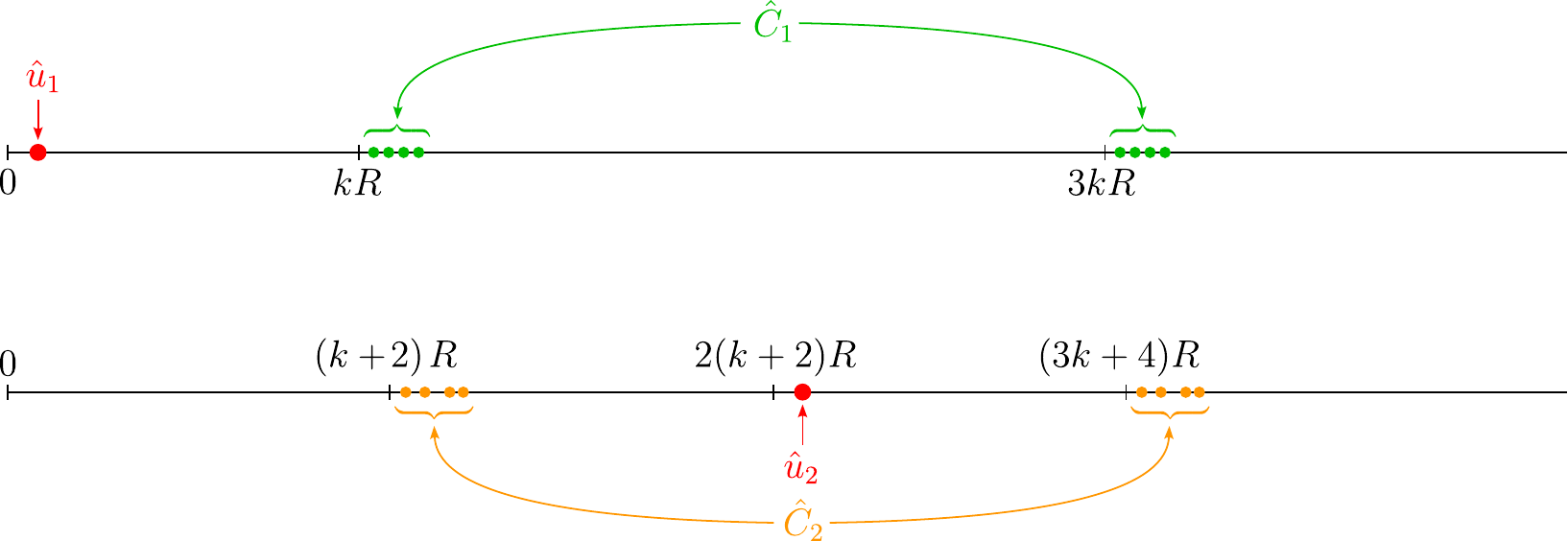}}
\end{center}
\caption{A schematic illustration of the construction in \Cref{lem:biexistd1}.} 
\label{fig:section5}
\end{figure}

\section{The lower bound in \texorpdfstring{\cref{thm:main}}{Theorem \ref{thm:main}}: \texorpdfstring{$d \geq 2$}{d ≥ 2}} \label{sec:lowerdg2}
We now pick up where we left off at the end of \cref{sec:lower}.
\subsection{A further reduction}
In this section, we reduce \cref{lem:biexist} to a problem involving simpler distance functions.
\begin{definition}\label{def:check}
Let $d \geq 2$ be a positive integer. Given a quadruple \[\check V = (\check v_1, \check v_2, \check x, \check y) \in \RR^{d-1} \times \RR^{d-1} \times \RR \times \RR_{>0}\] and a point $\check c \in \RR^{d-1}$, let
\[\check D^1_{\check V}(\check c) = \sqrt{\check x^2 + \Abs{\check c - \check v_1}^2} - \check x \quad\text{and}\quad \check D^2_{\check V}(\check c) = \check y\Abs{\check c - \check v_2}^2.\]
For two sets $\check C_1,\check C_2 \subseteq \RR^{d-1}$, define $\check D_{\check V} \colon \check C_1 \sqcup \check C_2 \to \RR$ to be the function that equals $\check D_{\check V}^1$ on $\check C_1$ and equals $\check D_{\check 
V}^2$ on $\check C_2$.  Let $\check \Sigma_{\check V}^{\check C_1,\check C_2} \colon [\abs{\check C_1} + \abs{\check C_2}] \to \check C_1 \sqcup \check C_2$ be the function such that $\check D_{\check V} \circ \check \Sigma_{\check V}^{\check C_1,\check C_2}$ is increasing (if such a function exists). Let $\check \Psi(\check C_1,\check C_2)$ be the set of all possible $\check \Sigma_{\check V}^{\check C_1,\check C_2}$ as $\check V$ varies, and let $\check \psi(\check C_1,\check C_2) = \abs{\check \Psi(\check C_1,\check C_2)}$.
\end{definition}
As in \cref{def:hatpsi}, the first two functions in \cref{def:check} can be viewed as effective distance functions for a suitable infinite separation limit.  The setup of this infinite separation limit is as described in the statement of the following lemma.

\begin{lemma}\label{lem:thing}
Let $k \geq 0$ and $d \geq 2$ be integers, and let $\check C_1, \check C_2 \subseteq \RR^{d-1}$ be finite point sets.  For sufficiently large $R > 0$, we have $\hat \psi_k(\set{0} \times \check C_1, \set{0} \times R\check C_2) \geq \check \psi(\check C_1, \check C_2)$.
\end{lemma}
\begin{proof}
For each $\check V = (\check v_1,\check v_2,\check x,\check y) \in \RR^{d-1} \times \RR^{d-1} \times \RR \times \RR_{>0}$, define $\hat U = (((\check x, \check  v_1), (R^2/(2\check y), R\check v_2))$. Then, for $\check c \in \check C_1$, we have
\[\hat D^1_{k, \hat U}((0, \check c)) = \sqrt{\Abs{\check c - \check v_1}^2+x^2} + \frac{R^2}{2\check y} = \check D^1_{\check V}(\check c) + \frac{R^2}{2\check y} + \check x.\]
For $\check c \in \check C_2$, we have
\begin{multline*}
\hat D^2_{k, \hat U}((0, R\check c)) = \sqrt{R^2\Abs{\check c - \check v_2}^2 + \frac{R^4}{4\check y^2}} + \check x = \frac{R^2}{2\check y} \sqrt{1 + \frac{4\check y^2\Abs{\check c - \check v_2}^2}{R^2}} + \check x \\
= \frac{R^2}{2\check y}\paren*{1 + \frac{4\check y^2\Abs{\check c - \check v_2}^2}{2R^2} + O(R^{-4})} + \check x 
= \check D_{\check V}^2(\check c) + \frac{R^2}{2\check y} + \check x + O(R^{-2}).
\end{multline*}
In particular, if $R$ is sufficiently large, then the relative order of the quantities $\hat D^1_{k, \hat U}((0, \check c))$ (for $\check c \in \check C_1$) and $\hat D^2_{k, \hat U}((0, R\check c))$ (for $\check c \in \check C_2$) is the same as the relative order of the quantities $\check D^1_{\check V}(\check c)$ (for $\check c \in \check C_1$) and $\check D^2_{\check V}(\check c)$ (for $\check c \in \check C_2$).  This concludes the proof.
\end{proof}
This lemma shows that in order to prove \cref{lem:biexist}, it remains only to find sets $\check C_1,\check C_2 \subseteq \RR^{d-1}$ each of size $m$ such that $\check\psi(\check C_1,\check C_2) = \Omega_{d}(m^{4d})$.

\subsection{The final construction}
Since the remainder of the argument is concerned with only the ``check'' variables and functions, we will dispense with diacritics except for $\check \psi$, $\check \Psi$, $\check D_V^i$, and $\check \Sigma_V^{C_1,C_2}$.

The goal of this section is to prove the following lemma.
\begin{lemma} \label{lem:psicc}
For $d \geq 2$ and $C \subseteq \RR^{d-1}$ finite, we have $\check \psi(C, C) \geq \binom{\abs{C}}{2}^2 \binom{\psi_1(C)}{2}$.
\end{lemma}
\cref{lem:biexist} can now be deduced by 
choosing $C$ to be some $m$-element subset of $\RR^{d-1}$ with $\psi_1(C) = \Theta_d(m^{2(d-1)})$; for the existence of such a set $C$, see the discussion in \Cref{sec:intro}.  For the rest of this section, fix a choice of $C \subseteq \RR^{d-1}$ ($d\geq 2$).

We start with some preliminary results that will help us understand $\check D^1_V$.  For $a > 0$, define the function $\vartheta_a \colon \RR \to \RR$ as $\vartheta_a(x) = \sqrt{x^2+a^2}-x$. For $a, b > 0$, define $\vartheta_{a,b}\colon \RR \cup \set{-\infty,+\infty} \to \RR$ by
\[\vartheta_{a,b}(x) = \begin{cases}
\vartheta_a(x)/\vartheta_b(x) & x \in \RR \\
1 & x = -\infty \\
a^2/b^2 & x = +\infty.
\end{cases}\]
We record several facts about $\vartheta_{a,b}$ that will be useful in the sequel.
\begin{lemma} \label{lem:theta}
Let $a > b > 0$.  Then $\vartheta_{a,b}$ is a continuous and strictly increasing function. Also, if $p \geq q > 0$ and $p/q \leq a^2/b^2$, then the unique $x \in \RR \cup \set{\pm\infty}$ satisfying $\vartheta_a(x)/\vartheta_b(x)= p/q$ is determined up to a sign by
\[x^2 = \frac{(a^2q^2-b^2p^2)^2}{4pq(p-q)(a^2q-b^2p)}.\]
(In particular, when $x = \pm \infty$, the numerator of the above fraction is nonzero, and the denominator is zero.)
\end{lemma}

\begin{proof}
It is clear that $\vartheta_a(x)$ is a smooth and positive function of $x$, so $\vartheta_a(x)/\vartheta_b(x)$ is also smooth and positive. To show that $\vartheta_a(x)/\vartheta_b(x)$ is strictly increasing on $\RR$, we start by noting that
\[\vartheta'_a(x) = \frac{x}{\sqrt{x^2 + a^2}} - 1 = -\frac{\vartheta_a(x)}{\sqrt{x^2+a^2}}.\]
Consequently,
\[\frac{d}{dx} \log \vartheta_a(x) = \frac{\vartheta'_a(x)}{\vartheta_a(x)} = -\frac{1}{\sqrt{x^2+a^2}}.\]
It follows that
\[\frac{d}{dx} \log \frac{\vartheta_a(x)}{\vartheta_b(x)} = \frac{1}{\sqrt{x^2+b^2}} - \frac{1}{\sqrt{x^2+a^2}} > 0,\]
which shows that $\vartheta_a(x)/\vartheta_b(x)$ is strictly increasing.

We now establish that $\vartheta_{a,b}$ is continuous at $x=\pm \infty$.  Combined with $\vartheta_{a,b}$ being strictly increasing on $\RR$, this also implies that $\vartheta_{a,b}$ is strictly increasing on its entire domain. In the limit $x \to -\infty$, we have $\vartheta_a(x)=(2+o_a(1))x$ and likewise for $\vartheta_b(x)$, so $\lim\limits_{x \to -\infty} \vartheta_a(x)/\vartheta_b(x) = 1$. In the limit $x \to \infty$, we Taylor-expand \[\vartheta_a(x)=x\paren*{\sqrt{1+(a/x)^2}-1}=x\paren*{\frac{1}{2} (a/x)^2+O((a/x)^4)} = \frac{a^2}{2x} + O_a(x^{-3}),\]
(and likewise for $\vartheta_b(x)$), and we see that $\lim\limits_{x \to \infty} \vartheta_a(x)/\vartheta_b(x) = a^2/b^2$.

It remains to derive the final formula of the lemma. The cases when $p/q \in \set{1,a^2/b^2}$ can be easily checked by hand. Now suppose that $x \in \RR$ and $\vartheta_a(x)/\vartheta_b(x)= p/q \in (1, a^2/b^2)$, which is equivalent to
\[q\sqrt{x^2 + a^2} - p\sqrt{x^2 + b^2} = (q-p)x.\]  Squaring both sides gives
\[(p^2+q^2)x^2 + a^2q^2 + b^2p^2 - 2pq\sqrt{(x^2+a^2)(x^2+b^2)} = (p-q)^2 x^2,\]
which can be rearranged to
\[2pq\sqrt{(x^2+a^2)(x^2+b^2)} = 2pqx^2 + (a^2q^2 + b^2p^2).
\]
Squaring again, we get
\[4p^2q^2 x^4 + 4(a^2 + b^2)p^2q^2x^2 + 4a^2b^2p^2q^2 = 4p^2q^2x^4 + 4(a^2q^2+b^2p^2)pqx^2 + (a^2q^2 + b^2p^2)^2,\]
which simplifies to
\[4pq(p-q)(a^2q-b^2p) x^2 = (a^2q^2-b^2p^2)^2.\]
Hence
\[x^2 = \frac{(a^2q^2-b^2p^2)^2}{4pq(p-q)(a^2q-b^2p)}.\]
(Note that $a^2q- b^2p\neq 0$ since $p/q < a^2/b^2$.)
\end{proof}

We now turn to the expressions appearing in $\check D^2_V$.  For $v \in \RR^{d-1}$, let \[\Delta(v) = \set{\Abs{v - c_1}^2/\Abs{v - c_2}^2 : c_1, c_2 \in C, \Abs{v - c_1} > \Abs{v - c_2}}.\] 
Say a pair of points $(v_1, v_2) \in (\RR^{d-1})^2$ is \dfn{good} if the following conditions hold:
\begin{itemize}
    \item $v_1$ and $v_2$ are not elements of $C$.
    \item $\Delta(v_1)$ and $\Delta(v_2)$ are disjoint and each have size $\binom{\abs{C}}{2}$.
    \item If $x \in \RR$ and $c_1, c_2, c_3, c_4 \in C$ are such that $(c_1,c_2) \neq (c_3,c_4)$, $\Abs{v_1 - c_1} > \Abs{v_1 - c_2}$, and $\Abs{v_1 - c_3} > \Abs{v_1 - c_4}$, then $\vartheta_{\Abs{v_1 - c_1},\Abs{v_1 - c_2}}(x)$ and $\vartheta_{\Abs{v_1 - c_3},\Abs{v_1 - c_4}}(x)$ are not both in $\Delta(v_2)$. 
\end{itemize}
The following lemma says that generic pairs are good; actually verifying all of the conditions for goodness is unfortunately somewhat tedious and notation-heavy.
\begin{lemma} \label{lem:denseopen}
The set of good pairs is dense and open in $(\RR^{d-1})^2$.
\end{lemma}
\begin{proof}
Let $C=\set{c_1, \ldots, c_m}$.  Define the following subsets of $(\RR^{d-1})^2$:
\begin{itemize}
    \item Let $\Omega$ be the set of $(v_1,v_2)$ such that $v_1,v_2$ do not lie in $C$.
    \item For distinct $\alpha,\beta \in [m]$, let $\Omega^{(1)}_{\alpha\beta} \subseteq \Omega$ be the set of $(v_1,v_2)$ such that $\Abs{v_1 - c_\alpha} \neq \Abs{v_1 - c_\beta}$.
    \item For distinct $\alpha,\beta \in [m]$, let $\Omega^{(2)}_{\alpha\beta} \subseteq \Omega$ be the set of $(v_1,v_2)$ such that $\Abs{v_2 - c_\alpha} \neq \Abs{v_2 - c_\beta}$.
    \item For $\alpha,\beta,\gamma,\delta\in [m]$ with $\alpha \neq \beta$, $\gamma \neq \delta$, and $(\alpha,\beta)\neq(\gamma,\delta)$, let $\Omega^{(3)}_{\alpha\beta\gamma\delta} \subseteq \Omega$ be the set of $(v_1,v_2)$ such that $\frac{\Abs{v_1 - c_\alpha}^2}{\Abs{v_1 - c_\beta}^2} \neq \frac{\Abs{v_1 - c_\gamma}^2}{\Abs{v_1 - c_\delta}^2}$.
    \item For $\alpha,\beta,\gamma,\delta\in [m]$ with $\alpha \neq \beta$, $\gamma \neq \delta$, and $(\alpha,\beta)\neq(\gamma,\delta)$, let $\Omega^{(4)}_{\alpha\beta\gamma\delta} \subseteq \Omega$ be the set of $(v_1,v_2)$ such that $\frac{\Abs{v_2 - c_\alpha}^2}{\Abs{v_2 - c_\beta}^2}\neq \frac{\Abs{v_2 - c_\gamma}^2}{\Abs{v_2 - c_\delta}^2}$.
    \item For $\alpha,\beta,\gamma,\delta\in [m]$ with $\alpha \neq \beta$ and $\gamma \neq \delta$, let $\Omega^{(5)}_{\alpha\beta\gamma\delta} \subseteq \Omega$ be the set of $(v_1,v_2)$ such that $\frac{\Abs{v_1 - c_\alpha}^2}{\Abs{v_1 - c_\beta}^2} \neq \frac{\Abs{v_2 - c_\gamma}^2}{\Abs{v_2 - c_\delta}^2}$.
    \item For $\alpha,\beta,\gamma,\delta,\eps,\zeta,\eta,\theta \in [m]$ with $\alpha \neq \beta$, $\gamma \neq \delta$, $\eps \neq \zeta$, $\eta \neq \theta$, and $\set{\alpha,\beta} \neq \set{\gamma,\delta}$, let $\Omega^{(6)}_{\alpha\beta\gamma\delta\eps\zeta\eta\theta} \subseteq \Omega$ be the set of $(v_1,v_2)$ such that there does not exist an $x \in \RR \cup \set{\pm \infty}$ with
    \[ \vartheta_{\Abs{v_1-c_\alpha},\Abs{v_1-c_\beta}}(x) = \frac{\Abs{v_2 - c_\eps}^2}{\Abs{v_2 - c_\zeta}^2}\quad\text{and}\quad\vartheta_{\Abs{v_1-c_\gamma},\Abs{v_1-c_\delta}}(x) = \frac{\Abs{v_2 - c_\eta}^2}{\Abs{v_2 - c_\theta}^2}.\tag{$\divideontimes$}\label{eq:omega-6}\]
\end{itemize}

We first claim that the intersection of all these sets is precisely the set of good pairs. First of all, the first condition in the definition of a good pair is precisely the condition that the pair is in $\Omega$; we will henceforth work within this set. Since at most one of $\frac{\Abs{v_1 - c_\alpha}^2}{\Abs{v_1 - c_\beta}^2}$ or $\frac{\Abs{v_1 - c_\beta}^2}{\Abs{v_1 - c_\alpha}^2}$ can be greater than $1$, the set $\Delta(v_1)$ has size $\binom{m}{2}$ if and only if the $m$ numbers $\Abs{v_1 - c_\alpha}^2$ for $\alpha \in [m]$ are all distinct and nonzero and their non-unit ratios are all distinct; the latter condition is exactly equivalent to $(v_1,v_2)$ lying in $\Omega^{(1)}_{\alpha\beta}$ and $\Omega^{(3)}_{\alpha\beta\gamma\delta}$ for all $\alpha,\beta,\gamma,\delta$. Similarly, the condition $\abs{\Delta(v_2)} = \binom{m}{2}$ is equivalent to $(v_1,v_2)$ lying in $\Omega^{(2)}_{\alpha\beta}$ and $\Omega^{(4)}_{\alpha\beta\gamma\delta}$ for all $\alpha,\beta,\gamma,\delta$. Now, assuming that $\abs{\Delta(v_1)} = \abs{\Delta(v_2)} = \binom{m}{2}$, the condition that $\Delta(v_1) \cap \Delta(v_2) = \varnothing$ is equivalent to $(v_1,v_2)$ lying in $\Omega^{(5)}_{\alpha\beta\gamma\delta}$ for all $\alpha,\beta,\gamma,\delta$, since as $\alpha,\beta,\gamma,\delta$ vary, the sets of values attained by the quantities $\frac{\Abs{v_1 - c_\alpha}^2}{\Abs{v_1 - c_\beta}^2}$ and $\frac{\Abs{v_2 - c_\gamma}^2}{\Abs{v_2 - c_\delta}^2}$ are given by $\Delta(v_1) \cup \set{1/t : t \in \Delta(v_1)}$ and $\Delta(v_2) \cup \set{1/t : t \in \Delta(v_2)}$, respectively.

To finish justifying the claim, we must show that, under the assumption of the first two conditions in the definition of goodness, the third condition also holds if and only if $(v_1,v_2)$ lies in all of the sets $\Omega^{(6)}_{\alpha\beta\gamma\delta\eps\zeta\eta\theta}$. The ``if'' direction is clear. For the converse, suppose $x \in \RR \cup \set{-\infty, \infty}$ and the indices $\alpha,\beta,\gamma,\delta,\eps,\zeta,\eta,\theta$ satisfy \eqref{eq:omega-6}.
We cannot have $x=\pm\infty$ since we assumed that the second condition of goodness holds. After possibly swapping $(\alpha,\eps)$ with $(\beta,\zeta)$ and/or swapping $(\gamma,\eta)$ with $(\delta,\theta)$, we may assume that $\frac{\Abs{v_2 - c_\eps}^2}{\Abs{v_2 - c_\zeta}^2}$ and $\frac{\Abs{v_2 - c_\eta}^2}{\Abs{v_2 - c_\theta}^2}$ are greater than $1$.  Then both of these quantities lie in $\Delta(v_2)$, so $(v_1,v_2)$ is not good.

Since finite intersections of dense and open subsets are dense and open, it suffices to show that each of the sets defined above is dense and open. This is obvious for $\Omega$. Moreover, since open subsets of open subspaces are open and dense subsets of dense subspaces are dense, it suffices to show that the rest of the sets are dense and open subsets of $\Omega$.

The openness of $\Omega^{(1)}_{\alpha\beta}$, $\Omega^{(2)}_{\alpha\beta}$, $\Omega^{(3)}_{\alpha\beta\gamma\delta}$, $\Omega^{(4)}_{\alpha\beta\gamma\delta}$, and $\Omega^{(5)}_{\alpha\beta\gamma\delta}$ is clear, as each can be easily expressed as the preimage of the open set $\RR^2 \setminus \set{(x,x):x \in \RR} \subseteq \RR^2$ under a suitable continuous map. To see that $\Omega^{(6)}_{\alpha\beta\gamma\delta\eps\zeta\eta\theta}$ is open, consider the set $\tilde \Omega \subseteq \Omega \times (\RR \cup \set{\pm\infty})$ (we omit indices for brevity) consisting of the triples $(v_1,v_2, x)$ satisfying \eqref{eq:omega-6}.
This is a closed subset since $\vartheta_{a,b}$ is continuous. Also, $\Omega^{(6)}_{\alpha\beta\gamma\delta\eps\zeta\eta\theta}$  is the complement of the image of $\tilde \Omega$ under the projection map $\Omega \times (\RR \cup \set{\pm\infty}) \to \Omega$. The openness of $\Omega^{(6)}_{\alpha\beta\gamma\delta\eps\zeta\eta\theta}$ now follows from the well-known fact that for topological spaces $X$ and $Y$ with $Y$ compact, the projection map $X \times Y \to X$ is closed.

We now check denseness.  This is obvious for $\Omega^{(1)}_{\alpha\beta}$ and $\Omega^{(2)}_{\alpha\beta}$. For $\Omega^{(3)}_{\alpha\beta\gamma\delta}$, it suffices to show that
\[\Abs{v_1 - c_\alpha}^2\Abs{v_1 - c_\delta}^2 - \Abs{v_1 - c_\beta}^2\Abs{v_1 - c_\gamma}^2\]
is not the zero polynomial in $v_1$. To see this, note that setting $v_1=c_\alpha$ gives $\Abs{c_\alpha-c_\beta}^2\Abs{c_\alpha-c_\gamma}^2$, which is nonzero unless $\alpha = \gamma$. Similarly, setting $v_1 = c_\beta$ yields a nonzero value unless $\beta = \delta$. But these cannot both occur due to our assumption that $(\alpha, \beta) \neq (\gamma, \delta)$. The argument for $\Omega^{(4)}_{\alpha\beta\gamma\delta}$ is identical. To show that $\Omega^{(5)}_{\alpha\beta\gamma\delta}$ is dense, it suffices to check that
\[\Abs{v_1 - c_\alpha}^2\Abs{v_2 - c_\delta}^2 - \Abs{v_1 - c_\beta}^2\Abs{v_2 - c_\gamma}^2\]
is not the zero polynomial in $v_1,v_2$; this is true since setting $v_2 = c_\gamma$ yields $\Abs{c_\gamma - c_\delta}^2 \Abs{v_1 - c_\alpha}^2$, which is clearly nonzero.

It remains to handle $\Omega^{(6)}_{\alpha\beta\gamma\delta\eps\zeta\eta\theta}$. If $(v_1,v_2,x)$ satisfies \eqref{eq:omega-6}, then two applications of the last part of \cref{lem:theta} yield that
\[\frac{(r_{1\alpha}r_{2\zeta}^2-r_{1\beta}r_{2\eps}^2)^2}{r_{2\eps}r_{2\zeta}(r_{2\eps}-r_{2\zeta})(r_{1\alpha}r_{2\zeta}-r_{1\beta}r_{2\eps})} = \frac{(r_{1\gamma}r_{2\theta}^2-r_{1\delta}r_{2\eta}^2)^2}{r_{2\eta}r_{2\theta}(r_{2\eta}-r_{2\theta})(r_{1\gamma}r_{2\theta}-r_{1\delta}r_{2\eta})},\]
where we abbreviate $r_{i\mu} = \Abs{v_i - c_\mu}^2$. Clearing denominators (which is still fine when $x = \pm \infty$), we get
\begin{multline*}
r_{2\eta}r_{2\theta}(r_{2\eta}-r_{2\theta})(r_{1\gamma}r_{2\theta}-r_{1\delta}r_{2\eta})(r_{1\alpha}r_{2\zeta}^2-r_{1\beta}r_{2\eps}^2)^2 \\ = r_{2\eps}r_{2\zeta}(r_{2\eps}-r_{2\zeta})(r_{1\alpha}r_{2\zeta}-r_{1\beta}r_{2\eps})(r_{1\gamma}r_{2\theta}^2-r_{1\delta}r_{2\eta}^2)^2.
\end{multline*}
It will suffice to show that these are not equal as polynomials in $v_1,v_2$.  Assume instead that these two polynomials are equal. The set of $v_2$ such that the left-hand side is zero for all $v_1$ is precisely the set consisting of $c_\eta$, $c_\theta$, and the points equidistant from $c_\eta$ and $c_\theta$.\footnote{Since $\alpha \neq \beta$, the polynomials $r_{1\alpha}$ and $r_{1\beta}$ are linearly independent, so no nontrivial linear combination of them is the zero polynomial. Applying similar reasoning to $r_{1\gamma}$ and $r_{1\delta}$, we conclude that $(r_{1\gamma}r_{2\theta}-r_{1\delta}r_{2\eta})(r_{1\alpha}r_{2\zeta}^2-r_{1\beta}r_{2\eps}^2)^2$ can never be the zero polynomial in $v_1$ for any value of $v_2$.} Analyzing the right-hand side likewise, we conclude that $\set{c_\eta, c_\theta} = \set{c_\eps, c_\zeta}$.  Without loss of generality, we have $\eps = \eta$ and $\zeta = \theta$. Now, after canceling factors that appear on both sides of our equation, we get
\[(r_{1\gamma}r_{2\zeta}-r_{1\delta}r_{2\eps})(r_{1\alpha}r_{2\zeta}^2-r_{1\beta}r_{2\eps}^2)^2 = (r_{1\alpha}r_{2\zeta}-r_{1\beta}r_{2\eps})(r_{1\gamma}r_{2\zeta}^2-r_{1\delta}r_{2\eps}^2)^2.\]
Plugging in $v_2 = c_\eps$ yields
\[\Abs{c_\eps - c_\zeta}^{10} r_{1\gamma}r_{1\alpha}^2 = \Abs{c_\eps - c_\zeta}^{10} r_{1\alpha}r_{1\gamma}^2.\]
So $\alpha = \gamma$. Similarly, plugging in $v_2 = c_\zeta$ implies $\beta = \delta$. This yields the desired contradiction.
\end{proof}

For the next step, we fix $v_1, v_2 \in \RR^{d-1}$ and consider $\check \Sigma^{C,C}_{(v_1,v_2,x,y)}$ as $x$ and $y$ vary. Let us write $\Gamma(v_1,v_2)= \abs{\set{(a,b)\in \Delta(v_1) \times \Delta(v_2) : a > b}}$.
\begin{lemma} \label{lem:gdd}
If $(v_1, v_2)$ is good, then as $V$ varies over $\set{(v_1,v_2)} \times \RR \times \RR_{>0}$, there are at least $\Gamma(v_1, v_2)$ possibilities for the function $\check \Sigma^{C,C}_V$.
\end{lemma}
\begin{proof}
Let $N = \Gamma(v_1,v_2)$. Let $\Xi$ be the set of quadruples $(c_1, c_2, c_3, c_4) \in C^4$ such that
\[1 < \frac{\Abs{v_2 - c_3}}{\Abs{v_2 - c_4}} < \frac{\Abs{v_1 - c_1}}{\Abs{v_1 - c_2}}.\]
For $i\in\set{1,2}$, each element of $\Delta(v_i)$ is associated with a unique ordered pair of elements in $C$ since $\abs{\Delta(v_i)}=\binom{\abs{C}}{2}$. So $\abs{\Xi}=N$.  \cref{lem:theta} tells us that for $\xi=(c_1,c_2,c_3,c_4) \in \Xi$, there is a unique $x_{\xi} \in \RR$ such that
\[\frac{\vartheta_{\Abs{v_1 - c_1}}(x_{\xi})}{\vartheta_{\Abs{v_1 - c_2}}(x_{\xi})} = \frac{\Abs{v_2 - c_3}^2}{\Abs{v_2 - c_4}^2}.\]
Moreover, since $(v_1,v_2)$ is good, these $x_{\xi}$'s are all distinct. Now order the elements of $\Xi$ as $\xi^{(1)}, \ldots, \xi^{(N)}$ so that
\[x_{\xi^{(1)}}<\cdots<x_{\xi^{(N)}}.\]
Write $\xi^{(j)}=(c_1^{(j)}, c_2^{(j)},c_3^{(j)}, c_4^{(j)})$ and $x_j=x_{\xi^{(j)}}$.  Choose arbitrary real numbers $x'_1, \ldots, x'_N$ such that
$x_1 < x'_1 < x_2 < x'_2 < \cdots < x_N < x'_N$.
For each $j$, we have
\[\frac{\vartheta_{\Abs{v_1 - c^{(j)}_1}}(x'_j)}{\vartheta_{\Abs{v_1 - c^{(j)}_2}}(x'_j)} > \frac{\vartheta_{\Abs{v_1 - c^{(j)}_1}}(x_j)}{\vartheta_{\Abs{v_1 - c^{(j)}_2}}(x_j)} = \frac{\Abs{v_2 - c^{(j)}_3}^2}{\Abs{v_2 - c^{(j)}_4}^2}\]
by construction, so we can find some $y_j > 0$ such that
\[\vartheta_{\Abs{v_1 - c^{(j)}_2}}(x'_j) < y_j\Abs{v_2 - c^{(j)}_4}^2 < y_j\Abs{v_2 - c^{(j)}_3}^2 < \vartheta_{\Abs{v_1 - c^{(j)}_1}}(x'_j).\]

Let $V_j = (v_1,v_2,x'_j, y_j)$.  Note that $\check D^1_{V_j}(c) = \vartheta_{\Abs{v_1 - c}}(x'_j)$ and $\check D^2_{V_j}(c) = y_j \Abs{v_2 - c}^2$. So the previous string of inequalities reads
\[\check D^1_{V_j}(c^{(j)}_2) < \check D^2_{V_j}(c^{(j)}_4) < \check D^2_{V_j}(c^{(j)}_3)< \check D^1_{V_j}(c^{(j)}_1).\]
We claim that the $N$ functions $\check \Sigma_{V_j}^{C,C}$ are all distinct.  Indeed, for $j' < j$, we cannot have
\[\check D^1_{V_{j'}}(c^{(j)}_2) < \check D^2_{V_{j'}}(c^{(j)}_4) < \check D^2_{V_{j'}}(c^{(j)}_3)< \check D^1_{V_{j'}}(c^{(j)}_1)\]
since this would imply that
\[
\frac{\vartheta_{\Abs{v_1 - c^{(j)}_1}}(x'_{j'})}{\vartheta_{\Abs{v_1 - c^{(j)}_2}}(x'_{j'})} >  \frac{\Abs{v_2 - c^{(j)}_3}^2}{\Abs{v_2 - c^{(j)}_4}^2},\]
contrary to the fact that $x'_{j'} < x_j$. Thus $\check \Sigma^{C,C}_{V_{j'}} \neq \check \Sigma^{C,C}_{V_j}$ for all $j' < j$, completing the proof. 
\end{proof}

\begin{proof}[Proof of \cref{lem:psicc}]
For each $\sigma \in \Psi_1(C)$, let $X_\sigma=\set{v\in\RR^{d-1}:\Sigma_v^C = \sigma}$; this is an open and nonempty subset of $\RR^{d-1}$. Let $U$ be the set of tuples $(v_\tau)_{\tau \in \Psi_1(C)} \in (\RR^{d-1})^{\Psi_1(C)}$ such that $(v_\sigma, v_{\sigma'})$ is good for all distinct $\sigma, \sigma' \in \Psi_1(C)$.  Notice that $U$ is dense and open.  Since the set $\prod_{\sigma \in \Psi_1(C)} X_\sigma\subseteq (\RR^{d-1})^{\Psi_1(C)}$ is open, the set $U \cap \prod_{\sigma \in \Psi_1(C)} X_\sigma$ is nonempty. Thus, we can find a tuple $(v_\sigma)_{\sigma \in \Psi_1(C)}$ such that $v_\sigma \in X_\sigma$ for all $\sigma$ and $(v_\sigma, v_{\sigma'})$ is good for all $\sigma \neq \sigma'$.

Fix some $\sigma \neq \sigma'$, and consider the functions $\check \Sigma_{(v_\sigma, v_{\sigma'}, x, y)}^{C,C}$ for $x \in \RR$ and $y \in \RR_{>0}$. Each such function ``extends'' $\sigma$ and $\sigma'$, in the sense that $\check D^1_{(v_\sigma, v_{\sigma'}, x, y)} \circ \sigma$ and $\check D^2_{(v_\sigma, v_{\sigma'}, x, y)} \circ \sigma'$
are both increasing functions. In particular, $\sigma$ and $\sigma'$ can both be recovered from $\check \Sigma_{(v_\sigma, v_{\sigma'}, x, y)}^{C,C}$, so
    \[
        \check \psi(C, C) \geq \sum_{\sigma \neq \sigma'} \Gamma(v_\sigma, v_{\sigma'}) 
        = \frac{1}{2} \sum_{\sigma \neq \sigma'} (\Gamma(v_\sigma, v_{\sigma'}) + \Gamma(v_{\sigma'}, v_\sigma)) 
        = \frac{1}{2} \sum_{\sigma \neq \sigma'} \binom{\abs{C}}{2}^2 
        = \binom{\abs{C}}{2}^2 \binom{\psi_1(C)}{2},
    \]
    as desired.
\end{proof}

\section{Unlimited vantage points}\label{sec:unlimited}

In this section, we study the set $\Psi(C)=\bigcup_{k\geq 1}\Psi_k(C)$ of orderings of a set $C=\set{c_1,\ldots,c_n}\subseteq\RR^d$ of candidate points that can be obtained using an arbitrarily large (finite) multiset of vantage points. That is, $\Psi(C)$ is the set of all tuples $\Sigma_V^C$ that can be obtained by choosing a finite multiset $V$ of vantage points in $\RR^d$ that distinguishes the points of $C$. 

We first observe that the triangle inequality places a constraint on the tuples in $\Psi(C)$. Let us say a tuple $(x_1,\ldots,x_n)$ of points in $\RR^d$ is \dfn{protrusive} if for every $i\in[n-1]$, the point $x_{i+1}$ is not in the convex hull of $x_1,\ldots,x_i$ (so the convex hull of $x_1,\ldots,x_{i+1}$ ``protrudes'' out of the convex hull of $x_1,\ldots,x_i$). 

\begin{proposition}\label{prop:convexity}
Let $C=\set{c_1,\ldots,c_n}\subseteq\RR^d$. Every tuple in $\Psi(C)$ is protrusive.  
\end{proposition}

\begin{proof}
Let $V$ be a finite multiset of vantage points in $\RR^d$ that distinguishes the points in $C$, and let $\Sigma_V^C=(c_{\sigma(1)},\ldots,c_{\sigma(n)})$ be the witnessed ordering of $C$. Fix $i\in[n-1]$. We need to show that $c_{\sigma(i+1)}$ is not in the convex hull of $c_{\sigma(1)},\ldots,c_{\sigma(i)}$. Suppose otherwise. Then there exist nonnegative real numbers $\lambda_1,\ldots,\lambda_i$ such that $\sum_{\ell=1}^i\lambda_\ell=1$ and $c_{\sigma(i+1)}=\sum_{\ell=1}^i\lambda_\ell c_{\sigma(\ell)}$. We have \[D_V(c_{\sigma(i+1)})=\sum_{v\in V}\Abs{c_{\sigma(i+1)}-v}=\sum_{v\in V}\Abs*{\sum_{\ell=1}^i\lambda_\ell(c_{\sigma(\ell)}-v)}\leq\sum_{v\in V}\sum_{\ell=1}^i\lambda_\ell\Abs{c_{\sigma(\ell)}-v}=\sum_{\ell=1}^i\lambda_\ell D_V(c_{\sigma(\ell)}).\] This contradicts the fact that $D_V(c_{\sigma(1)})<\cdots<D_V(c_{\sigma(i)})<D_V(c_{\sigma(i+1)})$. 
\end{proof}

It is natural to ask if $\Psi(C)$ is exactly the set of protrusive orderings of $C$. The next theorem states that this is the case when $d=1$, but we will see later that it can fail to hold when $d=2$. 

\begin{theorem}\label{thm:protrusive_d_1}
If $C=\set{c_1,\ldots,c_n}\subseteq \RR$, then $\Psi(C)$ is the set of protrusive orderings of $C$. The number of such orderings is $2^{n-1}$. 
\end{theorem}

\begin{proof}
We may assume without loss of generality that $c_1<\cdots<c_n$. It is straightforward to check that a tuple $(c_{\sigma(1)},\ldots,c_{\sigma(n)})$ is protrusive if and only if there do not exist indices $i_1<i_2<i_3$ such that $\sigma(i_3)$ is strictly between $\sigma(i_1)$ and $\sigma(i_2)$. (In other words, this tuple is protrusive if and only if $\sigma$ avoids the patterns $132$ and $312$.) It is well known (and straightforward to prove) that the number of such permutations is $2^{n-1}$.

\cref{prop:convexity} tells us that every tuple in $\Psi(C)$ is protrusive. To prove the reverse containment, we proceed by induction on $n$. This is trivial when $n=1$, so we may assume $n\geq 2$. Suppose $(c_{\sigma(1)},\ldots,c_{\sigma(n)})$ is protrusive; we will show that this tuple is in $\Psi(C)$. The entry $\sigma(n)$ is either $1$ or $n$. Let us assume $\sigma(n)=n$; a virtually identical argument handles the case where $\sigma(n)=1$. Let $C'=C\setminus\set{c_n}$. Then $(c_{\sigma(1)},\ldots,c_{\sigma(n-1)})$ is protrusive, so we can use induction to see that there exists a multiset $V'$ of points in $\RR$ such that $\Sigma_{V'}^{C'}=(c_{\sigma(1)},\ldots,c_{\sigma(n-1)})$. This means that $D_{V'}(c_{\sigma(1)})<\cdots<D_{V'}(c_{\sigma(n-1)})$. Let $K$ be a large positive integer. Let $V$ be the multiset obtained by adding $K$ copies of the point $c_1$ and $K$ copies of the point $c_{n-1}$ to $V'$. For $1\leq i\leq n-1$, we have \[D_V(c_i)=D_{V'}(c_i)+K(c_i-c_1)+K(c_{n-1}-c_i)=D_{V'}(c_i)+K(c_{n-1}-c_1).\] Thus, $D_V(c_{\sigma(1)})<\cdots<D_V(c_{\sigma(n-1)})$. We also have \[D_V(c_n)=D_{V'}(c_n)+K(c_n-c_1)+K(c_n-c_{n-1})=D_{V'}(c_n)+K(c_{n-1}-c_1)+2K(c_n-c_{n-1}),\] so $D_V(c_{\sigma(n-1)})<D_V(c_n)=D_V(c_{\sigma(n)})$ if $K$ is sufficiently large. This shows that \[(c_{\sigma(1)},\ldots,c_{\sigma(n)})=\Sigma_V^C\in\Psi(C).\qedhere\] 
\end{proof}

We now turn our attention to the Euclidean plane. We begin by showing that the analogue of \cref{thm:protrusive_d_1} is false when $d=2$.  In particular, we construct an explicit set $C \subseteq \RR^2$ of size $6$ such that $\Psi(C)$ does not contain all of the protrusive orderings of $C$.  The set $C$ will consist of the vertices of an equilateral triangle together with three points placed near the midpoints of its edges but just outside of the triangle; see \cref{fig:6points}.  In particular, this $C$ will be in convex position, so all orderings will be protrusive, but we will show that $\Psi(C)$ does not contain all orderings of $C$.

\begin{figure}[tbp]
\begin{center}{\includegraphics[height=3cm]{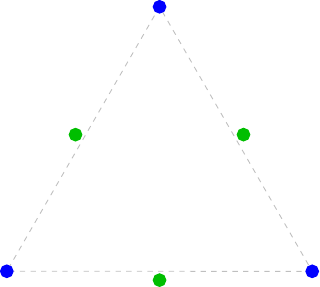}}
\end{center}
\caption{The set $C$ from \cref{prop:6-points}. The points $c_1,c_2,c_3$ are in blue, while $c_1',c_2',c_3'$ are in green. Note that the green points lie slightly outside of the triangle whose vertices are the blue points. }
\label{fig:6points}
\end{figure}

\begin{proposition}\label{prop:6-points}
Let $C \subseteq \RR^2$ consist of the six points
\begin{align*}
c_1&\coloneqq2e(\pi/2), &c_2&\coloneqq2e(7\pi/6),&c_3&\coloneqq2e(11\pi/6),\\
c'_1&\coloneqq-1.1e(\pi/2), &c'_2&\coloneqq-1.1e(7\pi/6), &c'_3&\coloneqq-1.1e(11\pi/6),
\end{align*}
where we have written $e(\theta)\coloneqq(\cos \theta, \sin \theta)$.
Then for any $v \in \RR$, we have
\[\Abs{v-c_1}+\Abs{v-c_2}+\Abs{v-c_3} \geq \Abs{v-c'_1}+\Abs{v-c'_2}+\Abs{v-c'_3}.\]
In particular, the ordering $(c_1, c_2, c_3, c'_1, c'_2, c'_3)$ is protrusive but is not in $\Psi(C)$.
\end{proposition}

\begin{proof}
The proof proceeds in two steps: We first show by a symmetry/variance argument that the proposition holds if $v$ is far from the origin, and then we check numerically that the proposition holds if $v$ is close to the origin.  Both parts of the proof are (relatively unenlightening) computations.

To begin, we first observe that for all $a, x > 0$, we have $\frac{\sqrt{a^2+x}-a}{x} = \frac{1}{2\sqrt{a^2+y}}$ for some $0 < y < x$ (by the mean value theorem). Therefore, we conclude that for all $a > 0$ and $x \geq 0$, we have
\[a + \frac{x}{2\sqrt{a^2+x}} \leq \sqrt{a^2 + x} \leq a + \frac{x}{2a}.\]

Now, for some $v$ with $\Abs{v} > 2$, let $L$ be the line through $v$ and the origin, and let $\pi_L \colon \RR^2 \to L$ denote the orthogonal projection onto $L$. Then, for any $c$ with $\Abs{c} \leq 2$, we have
\[\Abs{v - c} = \sqrt{\Abs{v - \pi_L(c)}^2 + \Abs{c - \pi_L(c)}^2};\]
this quantity can be bounded between
\[\Abs{v - \pi_L(c)} + \frac{\Abs{c - \pi_L(c)}^2}{2\Abs{v - c}} \quad \text{and} \quad \Abs{v - \pi_L(c)} + \frac{\Abs{c - \pi_L(c)}^2}{2\Abs{v - \pi_L(c)}}.\]
It follows that
\begin{align*}
    \MoveEqLeft \sum_{i \in [3]} \paren*{\Abs{v - c_i} - \Abs{v - c'_i}} \\ &\geq \sum_{i \in [3]} \paren*{\Abs{v - \pi_L(c_i)} - \Abs{v - \pi_L(c'_i)} + \frac{\Abs{c_i - \pi_L(c_i)}^2}{2\Abs{v - c_i}} - \frac{\Abs{c'_i - \pi_L(c'_i)}^2}{2\Abs{v - \pi_L(c'_i)}}} \\
    &\geq \sum_{i \in [3]} \paren*{\Abs{v - \pi_L(c_i)} - \Abs{v - \pi_L(c'_i)}} + \frac{\sum_{i \in [3]}\Abs{c_i - \pi_L(c_i)}^2}{2(\Abs{v}+2)} - \frac{\sum_{i \in [3]}\Abs{c'_i - \pi_L(c'_i)}^2}{2(\Abs{v}-1.1)}.
\end{align*}
We now analyze these three terms separately. First of all, since $c_1 + c_2 + c_3 = 0 = c'_1 + c'_2 + c'_3$ and $\Abs{v} \geq 2$, we get
\[\sum_{i \in [3]}\Abs{v - \pi_L(c_i)} = \Abs*{\sum_{i \in [3]}(v - \pi_L(c_i))} = \Abs*{\sum_{i \in [3]}(v - \pi_L(c'_i))} = \sum_{i \in [3]}\Abs{v - \pi_L(c'_i)},\]
so the first term is zero. Moreover, letting $u$ be a unit vector perpendicular to $L$, we claim that $\sum_{i \in [3]}(c_i \cdot u) c_i = 6u$. This is easy to verify for $u = c_j$ for some $j$, and thus follows for general $u$ from linearity. Therefore,
\[6 = \sum_{i \in [3]}(c_i \cdot u) (c_i \cdot u) = \sum_{i \in [3]} \Abs{c_i - \pi_L(c_i)}^2.\]
Similarly, $\sum_{i \in [3]} \Abs{c'_i - \pi_L(c'_i)}^2 = \frac{3}{2}\cdot 1.1^2$. Therefore,
\[\sum_{i \in [3]} \paren*{\Abs{v - c_i} - \Abs{v - c'_i}} \geq \frac{3}{4} \paren*{\frac{2^2}{\Abs{v}+2} - \frac{1.1^2}{\Abs{v}-1.1}};\]
this quantity will be positive whenever
\[\frac{\Abs{v} + 2}{\Abs{v} - 1.1} < \frac{2^2}{1.1^2},\]
which occurs whenever $\Abs{v} > 2.5$.

We now need to show that the inequality is true whenever $\Abs{v} \leq 2.5$, which can be done by brute force. Specifically, letting $f(v) = \sum_{i \in [3]} (\Abs{v - c_i} - \Abs{v - c'_i})$, we use a computer program with interval arithmetic to verify that on the $251 \times 251$ grid $v \in \set{-2.5,-2.48,\ldots,2.48,2.5}^2$, we always have $f(v) > 0.35$. Since $f$ is $6$-Lipschitz and no $v$ with $\Abs{v}\leq 2.5$ lies farther than $0.02/\sqrt{2}$ from a point on the grid, we conclude that
\[f(v) > 0.35 - 6 \cdot \frac{0.02}{\sqrt{2}} > 0.26\]
for all $v$ with $\Abs{v} \leq 2.5$.
\end{proof}

It is still worth studying the sets $C$ with the property that $\Psi(C)$ is precisely the set of protrusive orderings of $C$. The following theorems provide two families of sets $C$ with this property. 

\begin{theorem}\label{thm:4-points}
If $C\subseteq\RR^d$ is a set of size $n\leq 4$, then $\Psi(C)$ is the set of protrusive orderings of $C$.
\end{theorem} 

\begin{proof}
If the affine span of $C$ is $(n-1)$-dimensional, it is a simple exercise to show that every ordering of $C$ is in $\Psi_1(C)$.  Indeed, in this case, for each point $c \in C$, there is a sphere $S$ such that $c$ lies outside of $S$ and all of the other points of $C$ lie on $S$; adding many copies of the center of $S$ as vantage points does not change the relative order of $D_V$ on $C \setminus \set{c}$, but it increases the relative value of $D_V(c)$.  Henceforth we will assume the affine span of $C$ has dimension at most $n-2$. We may also assume that the points in $C$ are not collinear since, otherwise, the desired result follows from \cref{thm:protrusive_d_1}. So we are left to consider the case where $n=4$ and the affine span of $C$ is $2$-dimensional. Let $(c_1,c_2,c_3,c_4)$ be a protrusive ordering of $C$; we will show that this ordering is in $\Psi(C)$. It suffices to prove this when $d=2$. 

Let $C'=\set{c_1,c_2,c_3}$. We already know that the desired result holds for sets of $3$ candidate points, so there exists a multiset $V'$ of vantage points in $\RR^2$ such that $\Sigma_{V'}^{C'}=(c_1,c_2,c_3)$. We are going to find two points $p_1$ and $p_2$ such that \[\Abs{c_1-p_1}+\Abs{c_1-p_2}=\Abs{c_2-p_1}+\Abs{c_2-p_2}=\Abs{c_3-p_1}+\Abs{c_3-p_2}<\Abs{c_4-p_1}+\Abs{c_4-p_2}.\] If we can find such points, then we can form a new multiset $V$ by adding some $K$ copies of each of $p_1$ and $p_2$ to $V'$. If $K$ is sufficiently large, then our choice of $p_1$ and $p_2$ will ensure that $\Sigma_V^C=(c_1,c_2,c_3,c_4)$. 

If the points in $C'$ are collinear, then we can simply choose $p_1$ and $p_2$ to be two points such that the line segment with endpoints $p_1$ and $p_2$ contains $C'$. Now suppose the points in $C'$ are not collinear. We can find a point $z$ in the interior of the convex hull of $C$ such that $C' \cup \set{z}$ is in convex position. Since $c_1,c_2,c_3,z$ form the vertices of a convex quadrilateral, a well-known result \cite{Munn} implies that there is an ellipse $E$ passing through $c_1,c_2,c_3,z$. The point $c_4$ cannot lie in or on $E$ since that would imply that $z$, which is in the interior of the convex hull of $C$, is in the interior of $E$. Now, we can take $p_1$ and $p_2$ to be the foci of $E$. 
\end{proof}

In light of \cref{prop:6-points} and \cref{thm:4-points}, we are led to the following natural question. 

\begin{question}
Does there exist a set $C$ of $5$ points such that some protrusive ordering of $C$ is not in $\Psi(C)$?
\end{question}

We now formulate a sufficient condition on $C$ for $\Psi(C)$ to consist of all $\abs{C}!$ orderings of $C$.  If $C=\set{c_1, \ldots, c_n} \subseteq \RR^d$, then its \dfn{{distance matrix}} is defined to be the $n \times n$ matrix ${\mathbf M}_C$ whose $ij$-entry is $\Abs{c_i-c_j}$.  It is well known \cite{Schoenberg1937} (see also \cite{Ball1992}) that ${\mathbf M}_C$ is invertible for every choice of $C$.  Let $\mathbf{1}$ denote the all-$1$'s vector of length $n$.

\begin{lemma}\label{lem:dist-matrix}
Let $C=\set{c_1, \ldots, c_n} \subseteq \RR^d$.  If the vector $\nu\coloneqq{\mathbf M}_C^{-1}{\mathbf 1}$ has all entries strictly positive, then $\Psi(C)$ is the set of all $n!$ orderings of $C$, and every ordering is witnessed by a multiset of vantage points such that every vantage point is in $C$.
\end{lemma}

\begin{proof}
Assume without loss of generality that the diameter of the set $C$ is at most $1/(10n)$.  Let $c_{\sigma(1)}, \ldots, c_{\sigma(n)}$ be an ordering of $C$, and define the vector $\mu \in \RR^n$ coordinate-wise by $\mu_i\coloneqq\sigma^{-1}(i)$.  Then there is some constant $K>0$ such that the vector $\rho\coloneqq K\nu+{\mathbf M}_C^{-1}\mu$ has all entries strictly positive.  Notice that ${\mathbf M}_C \rho=K{\mathbf 1}+\mu$.  Obtain $\rho'$ from $\rho$ by rounding each coordinate down to the nearest integer.  The key observation is that since each entry of $\mathbf{M}_C$ is at most $1/(10n)$, each coordinate of ${\mathbf M}_C \rho'$ differs by at most $1/10$ from the corresponding coordinate of $\mathbf{M}_C \rho$.  In particular, the coordinates of ${\mathbf M}_C \rho'$, when ordered from smallest to largest, have the ordering $c_{\sigma(1)}, \ldots, c_{\sigma(n)}$.  Hence, if $V$ is the multiset consisting of $\rho'_i$ copies of each $c_i$, then $\Sigma_V^C=(c_{\sigma(1)}, \ldots, c_{\sigma(n)})$.
\end{proof}

It is straightforward to check that the hypothesis of \Cref{lem:dist-matrix} holds for many particular sets $C$.  One family of examples comes from taking $C$ to be the set of vertices of a vertex-transitive polytope; in this case, the vector $\nu={\mathbf M}_C^{-1}{\mathbf 1}$ is in fact constant.

\begin{theorem}\label{thm:polytope}
If $C$ is the set of vertices of a vertex-transitive polytope $P \subseteq \RR^d$, then $\Psi(C)$ is the set of all $\abs{C}!$ orderings of $C$.
\end{theorem}
\begin{proof}
Let $C=\set{c_1, \ldots, c_n}$.  The vertex-transitivity of $P$ ensures that the quantity $\sum_{j=1}^n \Abs{c_i-c_j}$ equals some constant $\gamma$ independent of $i$.  In particular, ${\mathbf M}_C \mathbf{1}=\gamma{\mathbf 1}$, so ${\mathbf M}_C^{-1}{\mathbf 1}=\gamma^{-1}{\mathbf 1}$ has all coordinates strictly positive.  The theorem now follows from \cref{lem:dist-matrix}.
\end{proof}

\section*{Acknowledgements}
Noga Alon was supported in part by NSF grant DMS-2154082. Colin Defant was supported by the National Science Foundation under Award No.\ 2201907 and by a Benjamin Peirce Fellowship at Harvard University. He thanks Rimma Hamaleinen for initially directing his attention to the article \cite{Carbonero2021}.  Noah Kravitz was supported in part by an NSF Graduate Research Fellowship (grant DGE-2039656).

\printbibliography
\end{document}